\documentclass[11pt,a4paper]{amsart}
\usepackage[margin=26mm,headheight=14pt]{geometry}
\usepackage[T1]{fontenc}
\usepackage{lmodern,microtype}
\usepackage{amsmath,amssymb,amsthm,mathtools,mathrsfs}
\usepackage{graphicx,booktabs,tabularx,array}
\usepackage{tikz}
\usetikzlibrary{arrows.meta,calc,patterns}
\usepackage[font=small,labelfont=bf]{caption}
\usepackage{fancyhdr}
\usepackage[colorlinks=true,linkcolor=blue!45!black,urlcolor=blue!45!black,citecolor=blue!45!black,bookmarksnumbered=true]{hyperref}
\hypersetup{pdftitle={Scalar curvature rigidity of small spherical caps with arbitrary fill-ins},pdfsubject={Publication manuscript}}
\newtheorem{theorem}{Theorem}[section]
\newtheorem{lemma}[theorem]{Lemma}
\newtheorem{proposition}[theorem]{Proposition}
\newtheorem{corollary}[theorem]{Corollary}
\theoremstyle{definition}

\theoremstyle{remark}
\newtheorem{remark}[theorem]{Remark}
\newcommand{\R}{\mathbb R}
\newcommand{\Sph}{\mathbb S}
\let\Cap\relax
\DeclareMathOperator{\Cap}{Cap}
\DeclareMathOperator{\tr}{tr}
\DeclareMathOperator{\Hess}{Hess}
\newcommand{\dd}{\,d}
\numberwithin{equation}{section}
\allowdisplaybreaks[2]
\title{Scalar curvature rigidity of small spherical caps\\with arbitrary fill-ins}

\author{Yiyue Zhang}
\author{Bowen Zhao}
\address{Beijing Institute of Mathematical Sciences and Applications,
Beijing 101408, China
}
\email{zhangyiyue@bimsa.cn}

\email{bowenzhao@bimsa.cn}

\begin{document}
\begin{abstract}
We prove an analogue of Min-Oo's conjecture for sufficiently small
spherical caps in the unit round sphere. For every dimension $n\ge3$,
there exists $r_*(n)>0$ such that the following holds for every round
cap of radius $0<r\le r_*(n)$: any smooth compact connected oriented
Riemannian fill-in with the same intrinsic boundary metric and with
scalar curvature and outward boundary mean curvature at least those
of the cap is isometric to the cap.
\end{abstract}
\maketitle

\setcounter{tocdepth}{1}

\section{Introduction}\label{sec:statement}

Let $(M^n,g)$ be a compact Riemannian manifold with entire boundary
$\Sigma:=\partial M$, and write $g_\Sigma$ for its induced boundary
metric. The spherical cap fill-in problem asks whether the intrinsic
metric of a round cap's boundary, together with the model lower bounds
for scalar curvature and outward mean curvature, determines the fill-in.
The boundary identification does not supply interior coordinates or
prescribe the topology of $M$.

For closed manifolds, Llarull's theorem \cite{Llarull} is a fundamental
scalar curvature rigidity result. A metric on $\Sph^n$ that dominates
the unit round metric and has scalar curvature at least $n(n-1)$ is
round. The same scalar curvature bound implies rigidity for a closed
spin manifold admitting a smooth area-nonincreasing map of nonzero degree
to the unit sphere. Here we impose comparison data only on the boundary,
without a spin structure or global metric comparison.

Min-Oo's conjecture \cite{MinOo} proposed that a metric on the hemisphere
with scalar curvature at least $n(n-1)$, unit round boundary metric,
and totally geodesic boundary must be round. Brendle--Marques--Neves
\cite{BMN} disproved it for every $n\ge3$, even with the metric fixed
on a boundary collar. Smaller caps behave differently. Brendle--Marques
\cite[Theorem~3]{BM} proved rigidity for sufficiently small round balls
under a $C^2$ closeness assumption.  Gromov records
the Brendle--Marques small cap conjecture in the setting of prescribed
boundary data \cite[\S3.12.2, item~II, pp.~162--163]{Gromov}: is there
a positive dimensional radius below which such rigidity holds for an
arbitrary fill-in? Corollary~\ref{cor:gromov} gives this conclusion.

Several earlier results establish rigidity with additional geometric
control. Hang--Wang \cite{HW} treated the conformal class of the
hemisphere. Spiegel \cite{Spiegel} treated locally conformally flat
fill-ins with umbilic round boundary. Hu--Liu--Shi \cite{HLS} proved
three-dimensional spherical cap rigidity under pointwise metric
domination, and Lee--Tam \cite{LT} obtained rigidity for spherical
domains under global metric domination in every dimension. None of
these extra assumptions is imposed below.

The problem also belongs to the theory of scalar curvature fill-ins.
Shi--Wang--Wei--Zhu \cite{SWWZ} established existence and obstruction
results, including obstructions from large boundary mean curvature.
Shi--Wang--Wei \cite{SWW} showed that an arbitrary boundary metric
extends to a positive scalar curvature metric when mean curvature is
unrestricted, underscoring the role of the latter constraint.
Cecchini--Hirsch--Zeidler \cite{CHZ} proved sharp mean curvature rigidity
for spin fill-ins, including a Euclidean ball equality case.
Frenck--Hanke--Hirsch \cite{FHH} obtained total mean curvature bounds by
surgery under specified dimension, spin, and mean curvature hypotheses.
B\"ar~\cite{Baer26} established upper bounds for the total mean curvature
of spin fill-ins under lower bounds for scalar curvature and boundary
mean curvature.
Our result concerns positive curvature spherical models and allows
arbitrary oriented fill-ins in every dimension $n\ge3$.

\subsection{Main result}

Write $\sigma$ for the unit round metric on $\Sph^{n-1}$. Mean curvature
is defined as the trace of the second fundamental form with respect to the outward unit normal; the Euclidean unit sphere therefore has mean curvature $n-1$ in our convention. A prescribed boundary isometry is understood in the notation
$g_\Sigma=\sigma$.

\begin{theorem}\label{thm:main}
For every integer $n\ge3$, there is a constant
$C_*(n)\in(0,n(n-1))$ with the following property. Let
$0<C\le C_*(n)$, and let $(M^n,g)$ be a smooth, compact, connected,
oriented Riemannian manifold with entire boundary $\Sigma=\partial M$.
If
\[
 g_\Sigma=\sigma,\qquad R_g\ge C,\qquad
 H_g\ge H_C:=(n-1)\sqrt{1-\frac{C}{n(n-1)}},
\]
then $(M,g)$ is isometric to the small round cap of scalar curvature
$C$ with boundary of intrinsic radius one.
\end{theorem}
\begin{remark}
One may take
\[
 C_*=2^{-128n^2}.
\]
We omit the detailed estimates of the constants to keep the presentation
concise.
\end{remark}
\begin{corollary}[Small cap rigidity]\label{cor:gromov}
For each $n\ge3$, there is $r_*(n)\in(0,\pi/2)$ such that the following
holds. Let $0<r\le r_*(n)$ and let $(M^n,g)$ be a smooth, compact,
connected, oriented Riemannian manifold with entire boundary $\Sigma$.
If
\[
 (\Sigma,g_\Sigma)\cong(\Sph^{n-1},\sin^2r\,\sigma),\qquad
 R_g\ge n(n-1),\qquad H_g\ge(n-1)\cot r,
\]
then $(M,g)$ is isometric to the geodesic ball of radius $r$ in the
unit round sphere.
\end{corollary}

\begin{proof}
Take $r_*(n)=\arcsin\sqrt{C_*(n)/[n(n-1)]}$ and apply
Theorem~\ref{thm:main} to $\widehat g=(\sin r)^{-2}g$.
Its boundary is unit round,
$R_{\widehat g}\ge n(n-1)\sin^2r$, and
$H_{\widehat g}\ge(n-1)\cos r$.
\end{proof}

The prescribed intrinsic boundary metric is part of the fill-ins
in Gromov's question. Without it, the two curvature inequalities alone
would also hold for smaller round balls. No interior topology, spin
structure, volume bound, or initial closeness to a round metric is
assumed in either statement.

\subsection{Proof strategy}

Attach the fixed Schwarzschild exterior to $(M,g)$ along $\Sigma$,
obtaining $(N,g_N)$ with nonnegative distributional scalar curvature.
Solving the conformal Laplacian equation with $u\to1$ at infinity
gives $\gamma:=u^{4/(n-2)}g_N$, which is scalar flat on both sides
of $\Sigma$ and has matching mean curvatures there. Its second
fundamental forms may still differ; smoothing and the positive mass
inequality give $m_\gamma\ge0$. Integrating the conformal equation,
with $J:=\int_Mu\,dV_g$, yields
\begin{equation}\label{eq:mainbudget}
 \begin{gathered}
 0\le2(n-1)|\Sph^{n-1}|m_\gamma
       \le C\left(\frac{|\Sph^{n-1}|}{n}-J\right),\\
 0\le m_\gamma\le\frac{C}{2n(n-1)}.
 \end{gathered}
\end{equation}
The key estimate of this paper is
\begin{equation}\label{eq:maincomparison}
 J\ge\frac{|\Sph^{n-1}|}{n}-K_nm_\gamma,
\end{equation}
where $K_n\ge1$ is the uniform
comparison constant fixed in Proposition~\ref{prop:mass-vanishing}.
Together, these inequalities force $m_\gamma=0$ when $C$ is
sufficiently small.

To preserve the prescribed boundary metric, solve the conformal
equation on $M$ with $u_D|_\Sigma=1$ and set
$h:=u_D^{4/(n-2)}g$. Then $R_h=0$, $h_\Sigma=\sigma$, and
$u_D/u$ is $\gamma$-harmonic. Section~\ref{sec:energy} extends a
constant multiple of this ratio $u_D/u$ to the exterior, producing a Euclidean extension of $h$. A Brown--York identity bounds the Dirichlet energy of this extended conformal factor by $m_\gamma$, e.g. \eqref{eq:energybound}.

The main comparison \eqref{eq:maincomparison} rests on two main ideas. First, in section \ref{sec:geometry} we use the quantitative minimal graph estimates underlying the stability argument of Wang--Zhang \cite{WZ} to control the geometry of $h$.
Because the Euclidean extension of $h$ need not have nonnegative distributional scalar curvature at $\Sigma$, the graphs are constructed over smooth approximations of $\gamma$ and the estimates are transferred to $h$ through the aforementioned conformal relation. Their derivative
estimates, together with the conformal energy bound and a separate
boundary identification argument, give coordinates in which $h$
is close to the Euclidean metric on a nearly round domain, apart
from a set of small capacity. The metric and boundary errors are
controlled by $\sqrt{m_\gamma}$ in the relevant energy norms,
up to arbitrarily small approximation errors.

Second, Section~\ref{sec:weighted} improves this control to an error
of order $m_\gamma$, yielding a sharp lower bound for
\[
 J=\int_Mu\,dV_g.
\]
We construct a small solution to a Dirichlet comparison problem and
expand the comparison to first order in the metric perturbation and
boundary displacement.
An identity involving the adjoint of the linearized scalar curvature
operator, together with $R_h=0$, converts the linear metric contribution
into a boundary term. The condition $h_\Sigma=\sigma$ relates this term
to the boundary displacement, causing the two linear contributions to
cancel. The remaining errors are quadratic and hence of order
$m_\gamma$.
A boundary mean identity and a Poisson test control the remaining
harmonic weight, completing \eqref{eq:maincomparison}. This improvement
from $\sqrt{m_\gamma}$ to $m_\gamma$ allows the mass inequality to force
$m_\gamma=0$.

When $m_\gamma=0$, the conformal energy vanishes, and Brown--York
equality identifies $(M,h)$ with the Euclidean ball. A positive
static potential then recovers the round cap metric $g$.
Thus ball topology follows from the equality argument.

The Riemannian positive mass theorem was first proved by Schoen--Yau
\cite{SY} using minimal surfaces. Brendle--Wang \cite{BW} developed a
new dimension descent scheme extending the Schoen--Yau approach to
arbitrary dimensions and addressing the singularities that arise in
the induction. We use the smooth positive mass theorem of Khuri--Wang--Wang
\cite{KWW} and the stability theorem of Wang--Zhang \cite{WZ}.
The required results are recalled in Section~\ref{sec:foundations}.

\subsection{Notation and organization}

We use the conventions
\[
 \Delta_g:=\operatorname{div}_g\nabla_g,\qquad
 L_g:=-\frac{4(n-1)}{n-2}\Delta_g+R_g.
\]
The notation $O_j(r^{-b})$ includes derivatives through order $j$, with one additional power of decay for each derivative. Fix $0<\theta_0<1$. The notation $O_{j,\theta_0}(r^{-b})$ also includes the corresponding scaled H\"older seminorms.

On the asymptotically flat manifolds considered below, we write
$H_0^1(N)$ for the completion of $C_c^\infty(N)$ in the Dirichlet
norm $\|d\phi\|_2$. The global Sobolev inequality gives a continuous
embedding into $L^{2n/(n-2)}(N)$; global $L^2$ integrability is not
required. This distinction is relevant in dimensions three and four,
where a function with decay of order $r^{2-n}$ may have finite
Dirichlet energy without being square integrable.
On bounded domains, $H_0^1$ has its usual zero-trace meaning.

Unless otherwise stated, $K$ denotes a positive constant depending
only on $n$ and fixed dimensional parameters, and may change from
line to line. Constants used in existence and approximation
arguments may depend on the fixed fill-in. The corresponding
approximation errors are sent to zero with the fill-in held fixed.

The principal metrics are summarized below; $g'$ denotes a generic
smooth metric used in the graph construction.
\begin{center}
\small
\begin{tabularx}{\textwidth}{@{}l >{\raggedright\arraybackslash}X@{}}
\toprule
Metric & Definition and role\\
\midrule
$g$ & Given fill-in metric on $M$.\\
\addlinespace
$g_N$ & $g$ joined to the fixed Schwarzschild exterior.\\
\addlinespace
$\gamma$ & $u^{4/(n-2)}g_N$; scalar flat on each side of $\Sigma$,
with matching mean curvatures.\\
\addlinespace
$h$ & $u_D^{4/(n-2)}g$ on $M$, extended by
$r_{\mathrm S}^{-2}v^{4/(n-2)}\gamma$; $h_\Sigma=\sigma$
and the exterior is Euclidean.\\
\addlinespace
$g'$ & Smooth metric for the graph estimates; later
$g'=\gamma_\varepsilon$, an approximation of $\gamma$.\\
\addlinespace
$\widetilde g_i$ & Conformally modified metric of the $i$th minimal graph.\\
\bottomrule
\end{tabularx}
\end{center}

Section~\ref{sec:foundations} presents the analytic inputs,
smoothing across $\Sigma$, and Brown--York positivity and rigidity.
Section~\ref{sec:energy} establishes the mass budget and the energy
identity relating the two conformal changes.
Section~\ref{sec:geometry} constructs the comparison coordinates
and identifies the prescribed boundary and the portion of the
fill-in represented by these coordinates.
Section~\ref{sec:weighted} proves the integral comparison through
the first variation cancellation and the harmonic weight estimate.
Finally, Section~\ref{sec:equality} completes the rigidity argument.

\medskip
\noindent\textbf{AI usage.} The authors made substantial use of AI in developing and editing this manuscript. The authors carefully checked all
arguments and polish the paper.

\medskip
\noindent\textbf{Acknowledgements.}  YZ is partially supported by NSFC
Grant No.\ 12501070. Both authors are partially supported by the startup fund from BIMSA. BZ gratefully acknowledges the Morningside Center, Chinese Academy of Sciences, for hospitality.

\section{Analytic foundations}
\label{sec:foundations}

We separate the external smooth manifold results from the constructions
needed at the gluing hypersurface. Our mass normalization is
\[
 m_{\mathrm{ADM}}(g')=\frac{1}{2(n-1)|\Sph^{n-1}|}
 \lim_{r\to\infty}\int_{S_r}
 (\partial_jg'_{ij}-\partial_ig'_{jj})\frac{x^i}{r}\,dA_\delta.
\]
In particular, the metric
$(1+b_{\mathrm{end}}r^{2-n}+O(r^{1-n}))^{4/(n-2)}\delta$ has mass $2b_{\mathrm{end}}$.

\subsection{The external positive mass and graph inputs}
We use the smooth version of the positive mass theorem of
Khuri--Wang--Wang \cite[Theorem~1.1]{KWW}. Specifically, let
$(N^n,g')$, $n\ge3$, be smooth, complete, connected, and without
boundary, with one asymptotically flat end and compact complement.
If
\[
 R_{g'}\ge0,\qquad R_{g'}\in L^1(N,g'),\qquad
 \partial^j(g'-\delta)=O(r^{-q-j}),\quad j=0,1,2,
 \qquad q>\frac{n-2}{2},
\]
then $m_{\mathrm{ADM}}(g')\ge0$.
The singular set is empty, and the decay satisfies
\cite[Definition~2.2]{KWW}. Our original ends have $q=n-2$.
 We use only the mass inequality and apply it to smooth
metrics after any required smoothing and conformal correction.

We use the quantitative minimal graph estimates of Wang--Zhang
\cite{WZ} in every dimension $n\ge3$. Their hypotheses are
smoothness, completeness, connectedness, orientation, absence of
boundary, finitely many asymptotically flat ends, nonnegative
integrable scalar curvature, and harmonic asymptotics on a
distinguished end. Our approximating metrics have one end,
compact complement, and the stronger exact form
\[
 g'=\phi^{4/(n-2)}\delta,\qquad
 \Delta_\delta\phi=0,\qquad
 \phi>0,\qquad \phi\longrightarrow1
\]
outside a compact set. The exterior harmonic expansion gives the
decay required in \cite[equations~(2.2) and~(2.4)]{WZ}.
No restriction on the topology of the compact core is imposed.

Fix $\lambda:=\lambda_0(n)/2>0$, where $\lambda_0(n)$ is the
dimensional slope bound in \cite[Theorem~3.1]{WZ}.
That theorem provides global minimal graph functions
$u_i=\lambda\Phi_i$, $i=1,\ldots,n$, with asymptotic slopes
$\lambda e_i$. The resulting map
\[
 \Phi:=(\Phi_1,\ldots,\Phi_n):N\longrightarrow\mathbb R^n
\]
satisfies
\[
 \Phi(x)=x+o(r),\qquad D\Phi=I+o(1)
\]
on the end, where $x$ denotes the asymptotic coordinates and
$r:=|x|$. The gradients of the graph functions are globally
bounded for each fixed metric $g'$; this bound need not be uniform
over all fill-ins.

The modified graph metrics
\[
 \widetilde g_i:=(1+|du_i|_{g'}^2)^{-1/(n-1)}
                    (g'+du_i\otimes du_i)
\]
are complete and asymptotically flat, with nonnegative integrable
scalar curvature \cite[equations~(4.4)--(4.6)]{WZ}.
Their masses satisfy
\[
 m_{\mathrm{ADM}}(\widetilde g_i)
   =(1+\lambda^2)^{1/[2(n-1)]}m_{\mathrm{ADM}}(g')
\]
by \cite[Proposition~5.1]{WZ}.

For each $\widetilde g_i$, \cite[Lemma~6.1]{WZ} provides a positive
conformal potential $\varphi_i$ satisfying
\[
 -\frac{4(n-1)}{n-2}\Delta_{\widetilde g_i}\varphi_i
       +R_{\widetilde g_i}\varphi_i=0,\qquad
 0<\varphi_i\le1,\qquad
 \varphi_i\longrightarrow1
\]
at infinity. The conformal mass formula and the global and localized
energy estimates are given in
\cite[Propositions~6.2--6.3 and Corollary~6.5]{WZ}.
For the coordinate defect defined in
\eqref{eq:marked-Q-definition}, these estimates and
\cite[Proposition~7.1 and equation~(7.7)]{WZ} yield
\eqref{eq:marked-WZ-energy}. The normalization
$\Phi_i=u_i/\lambda$ introduces a fixed factor $\lambda^{-2}$,
which is absorbed into the dimensional constant. The constants
and the admissible defect threshold depend only on $n$.

We use these local estimates without invoking the subsequent
global extraction theorem of Wang--Zhang.
Section~\ref{sec:geometry} constructs the coordinate chart and
comparison boundary needed here, including their compatibility
with the prescribed boundary $\Sigma$.

\subsection{The fixed exterior and the first conformal equation}

Define
\begin{equation}\label{eq:parameters}
 z=\frac{1+\sqrt{1-C/[n(n-1)]}}2,
 \qquad r_{\mathrm S}=z^{2/(n-2)}.
\end{equation}
Thus
\[
 C=4n(n-1)z(1-z),\qquad r_{\mathrm S}^{n-2}=z^2,
 \qquad H_C=(n-1)(2z-1).
\]
Attach to $(M,g)$ the Schwarzschild exterior
\[
 g_N=F^{4/(n-2)}\delta,\qquad
 F(r)=1+\frac{C}{4n(n-1)r^{n-2}},\qquad r\ge r_{\mathrm S}.
\]
Since $F(r_{\mathrm S})=z^{-1}$, its inner
boundary has intrinsic metric $\sigma$ and, with normal toward
increasing $r$, mean curvature $H_C$. The coordinate radius
$r_{\mathrm S}$ is therefore distinct from the intrinsic boundary
radius, which is one.

The boundary identification at $\Sigma$ gives a complete manifold $N$ (we will use $g_N$ to denote the full metric on $N$) with a
continuous metric in normal collar coordinates. The scalar curvature
measure is
\[
 d\mathcal R=R_g\,dV_g+2(H_g-H_C)\,d\sigma\ge0.
\]
The term supported on $\Sigma$ is the contribution of the jump in mean curvature.
Indeed, in a normal collar the identity
$R=R_{g_s}-|\mathrm{II}|^2-H^2-2\partial_sH$ gives this coefficient
and sign \cite[Section~2]{Miao}.

\begin{lemma}\label{lem:globalflatten}
There is a unique solution $u-1\in H_0^1(N)$, tending to one at
infinity, of
\[
 \frac{4(n-1)}{n-2}\int_N\langle du,d\phi\rangle_{g_N}\,dV_{g_N}
       +\int_Nu\phi\,d\mathcal R=0
 \qquad (\phi\in H_0^1(N)).
\]
It is positive, at most one, continuous across $\Sigma$, and smooth
up to $\Sigma$ from each side. With both normal derivatives pointing
from $M$ into the exterior, it satisfies
\[
 \frac{4(n-1)}{n-2}
       (\partial_{\nu,+}u-\partial_{\nu,-}u)
       =2(H_g-H_C)u.
\]
Consequently $\gamma=u^{4/(n-2)}g_N$ is complete, scalar flat on each
side, and has matching mean curvatures at $\Sigma$. Its exterior is
\[
 \gamma=w_\infty^{4/(n-2)}\delta,\qquad w_\infty=Fu,
 \qquad \Delta_\delta w_\infty=0,
\]
where
\[
 w_\infty=1+\frac{m_\gamma}{2r^{n-2}}
                    +O_{j,\theta_0}(r^{1-n})
\]
for every fixed finite $j$.
\end{lemma}

\begin{proof}
The bilinear form
\[
 \mathcal B(\phi,\psi)=
 \frac{4(n-1)}{n-2}\int_N\langle d\phi,d\psi\rangle_{g_N}\,dV_{g_N}
                 +\int_N\phi\psi\,d\mathcal R
\]
is bounded and coercive on $H_0^1(N)$. Boundedness of the compactly
supported volume and hypersurface terms follows from the fixed metric
Sobolev and trace inequalities; coercivity follows from
$d\mathcal R\ge0$. Lax--Milgram applied to $1-u$ gives existence and
uniqueness. Negative part and upper truncation tests give $0\le u\le1$.
Standard tangential difference quotients for the transmission equation,
followed by the one-sided equations to recover normal derivatives,
give smoothness up to the interface and a common continuous trace.

Integration by parts on the two pieces shows that the coefficient of
the interface test function is
\[
 \frac{4(n-1)}{n-2}
       (\partial_{\nu,-}u-\partial_{\nu,+}u)+2(H_g-H_C)u.
\]
Its vanishing is the asserted transmission condition. The sign of the
exterior derivative reflects its inner outward normal $-\nu$.
The exterior solution is nonzero because $u-1\in L^{2n/(n-2)}$;
therefore it is positive there. A zero on $\Sigma$ would give
$\partial_{\nu,+}u>0$ by Hopf's lemma and
$\partial_{\nu,-}u\le0$ by nonnegativity inside, contradicting the
transmission condition. The interior maximum principle now gives
$u>0$ everywhere.

Conformal covariance gives $\Delta_\delta(Fu)=0$ on the end.
The homogeneous Sobolev condition and harmonic estimates on annuli
imply $Fu\to1$. Exterior harmonic expansion then gives the stated
asymptotics; the coefficient is $m_\gamma/2$ by our mass convention.
Thus $u$ has a positive global lower bound for this fixed fill-in,
which proves completeness of $\gamma$. Scalar flatness follows from
the conformal equation. Finally,
\[
 H_{\gamma,-}-H_{\gamma,+}
 =u^{-n/(n-2)}\left[
 u(H_g-H_C)+\frac{2(n-1)}{n-2}
                 (\partial_{\nu,-}u-\partial_{\nu,+}u)\right]=0.
\]
\end{proof}

\subsection{Smoothing a corner with control of the connection}

Matching mean curvatures removes the scalar curvature measure at
$\Sigma$. It does not identify the full second fundamental forms, and thus
does not make the metric smooth. The next approximation supplies both
the smooth metrics needed by the external theorems and the connection
control needed to transfer their Hessian estimates.

\begin{lemma}\label{lem:corner}
Let $\gamma$ be a complete metric on a manifold without boundary,
with one asymptotically flat end having compact complement. Suppose
that $\gamma$ is continuous across a compact smooth hypersurface
and smooth up to it from each side. Suppose $R_\gamma\ge0$ on each side and $H_--H_+\ge0$, with
both mean curvatures computed toward the end. Suppose that, outside
a compact set,
\[
 \gamma=w_\infty^{4/(n-2)}\delta,\qquad
 \Delta_\delta w_\infty=0,\qquad w_\infty>0,\qquad
 w_\infty=1+\frac{m}{2r^{n-2}}+O(r^{1-n}).
\]
Then there are smooth complete metrics $\gamma_\varepsilon$ with
nonnegative scalar curvature and exactly harmonically flat ends
in the same chart, such that
\begin{equation}\label{eq:standalone-corner-limit}
 \begin{gathered}
 \sup_N|\gamma_\varepsilon-\gamma|_\gamma\longrightarrow0,
 \qquad m_{\gamma_\varepsilon}\longrightarrow m,\\
 \int_N|\nabla^{\gamma_\varepsilon}-\nabla^\gamma|_\gamma^2
                                      \,dV_\gamma\longrightarrow0.
 \end{gathered}
\end{equation}
In particular, $m\ge0$. The differentiable structure is fixed by
the one-sided Gaussian collars and unchanged away from $\Sigma$;
the connection difference is measured almost everywhere in this
structure. The approximation width may depend on the fixed metric.
\end{lemma}

\begin{proof}
We use the corner approximation of Miao~\cite[Section~3]{Miao}, in
the form of Lee--Lesourd--Unger~\cite[Theorem~3.2, Lemma~3.7,
and Remark~3.8]{LLU}. Their approximation and conformal correction
are independent of the dimension and of any spin assumption.
We record the additional connection estimate needed here.

The smooth metrics $\widehat\gamma_\varepsilon$ agree with
$\gamma$ outside a shrinking collar, converge uniformly, and have
uniformly bounded first derivatives in the fixed Gaussian collars.
Their negative scalar curvature is uniformly bounded and supported
in that collar. Consequently
\[
 \|\nabla^{\widehat\gamma_\varepsilon}-\nabla^\gamma\|_{L^2(\gamma)}
 \longrightarrow0.
\]
Choose smooth nonnegative functions
$f_\varepsilon\ge (R_{\widehat\gamma_\varepsilon})^-$,
supported in a slightly larger shrinking collar and satisfying
$\|f_\varepsilon\|_{L^p}\to0$ for every finite $p$. The conformal correction in \cite[Lemma~2.4 and the proof of
Theorem~3.3]{LLU} solves
\[
 -\frac{4(n-1)}{n-2}\Delta_{\widehat\gamma_\varepsilon}b_\varepsilon
       =f_\varepsilon b_\varepsilon,
 \qquad b_\varepsilon-1\in H_0^1(N),
\]
and gives $b_\varepsilon\ge1$ and
$\|b_\varepsilon-1\|_\infty\to0$. Testing the equation with
$b_\varepsilon-1$ and using the common Sobolev inequality gives
\[
 \|db_\varepsilon\|_2
 \le K(\gamma)\|f_\varepsilon\|_{2n/(n+2)}\longrightarrow0.
\]
Thus $\gamma_\varepsilon:=
b_\varepsilon^{4/(n-2)}\widehat\gamma_\varepsilon$ is smooth,
complete, and has nonnegative scalar curvature. Uniform convergence
and convergence of its ADM mass are part of the cited approximation.
The conformal connection formula gives
\[
 |\nabla^{\gamma_\varepsilon}
       -\nabla^{\widehat\gamma_\varepsilon}|_\gamma
 \le K(\gamma)b_\varepsilon^{-1}|db_\varepsilon|_\gamma,
\]
which proves the asserted connection convergence.

The correction is $\gamma$-harmonic near infinity. Conformal
covariance therefore gives $\Delta_\delta(w_\infty b_\varepsilon)=0$.
Harmonic expansion shows that the end remains exactly harmonically
flat, with the required differentiated asymptotics. Finally, applying \cite[Theorem~1.1]{KWW} to the
smooth metrics and taking the mass limit gives $m\ge0$.
\end{proof}

For the coordinate map $\bar\Phi$ of Section~\ref{sec:geometry}, on any
region where its differential is uniformly bounded,
\[
 \Hess_\gamma\bar\Phi
   =\Hess_{\gamma_\varepsilon}\bar\Phi
        +(\nabla^{\gamma_\varepsilon}-\nabla^\gamma)*d\bar\Phi.
\]
Thus the squared $L^2$ error in the Hessian estimate tends to zero.
For each fixed fill-in, we choose a sufficiently fine approximation;
the collar width need not be uniform across fill-ins.
We do not claim global $L^2$ convergence of the metric coefficients,
since the mass terms need not be square integrable.
\subsection{Brown--York positivity and rigidity for a round boundary}

\begin{lemma}\label{lem:BYpositive}\label{lem:EuclideanBYrigidity}
Let $(M^n,h)$, $n\ge3$, be smooth, compact, connected, and oriented,
with nonnegative scalar curvature, entire boundary $\Sigma$
isometric to the unit round sphere, and positive outward mean
curvature $H_h$. Then
\[
 \int_\Sigma(n-1-H_h)\,d\sigma\ge0.
\]
Moreover, if $R_h=0$ and $H_h=n-1$, then $(M,h)$ is isometric
to the closed Euclidean unit ball.
\end{lemma}

\begin{proof}
The Shi--Tam extension~\cite{ShiTam} attaches a scalar flat,
asymptotically flat exterior to $(M,h)$, with the same boundary
metric and mean curvature, whose ADM mass satisfies
\[
 m_{\mathrm{ADM}}
 \le \frac{1}{(n-1)|\Sph^{n-1}|}
            \int_\Sigma(n-1-H_h)\,d\sigma.
\]
For this formulation and its application to the glued manifold,
see \cite[Section~4, proof of Theorem~1.8]{LLU}.
The glued metric is complete and has nonnegative scalar curvature
away from the corner, where the mean curvatures match. Approximate
it using \cite[Theorem~3.2 and Remark~3.8]{LLU}. On the end, the
correcting factor is harmonic, and its difference from one belongs
to $W^{2,p}_{-q}$ for $p>n$ and $(n-2)/2<q<n-2$. The end estimates
of \cite[Lemma~2.5 and equation~(2.23)]{ShiTam}, together with scaled
Schauder estimates, give the pointwise decay required above.
Applying the smooth positive mass inequality
\cite[Theorem~1.1]{KWW} and passing to the limit of the ADM masses
gives $m_{\mathrm{ADM}}\ge0$, proving the boundary inequality.
Here we use the approximation theorem of \cite{LLU}, whose proof
has no dimension or spin restriction, and obtain positivity from
\cite{KWW}.

Now suppose $R_h=0$ and $H_h=n-1$. We follow the compact variation
argument of \cite[Section~4]{EMW}. For a smooth symmetric tensor
$\kappa$ supported in $M^\circ$, set $h_t:=h+t\kappa$ and solve
\[
 L_{h_t}b_t=0,\qquad b_t|_\Sigma=1.
\]
For small positive and negative $t$, this gives a smooth positive
family with $b_0=1$. The metrics $b_t^{4/(n-2)}h_t$ are scalar flat,
have boundary metric $\sigma$, and have positive mean curvature.
The boundary inequality just proved shows that their total boundary
mean curvature has a maximum at $t=0$. Differentiating at zero,
as in \cite[Lemma~4]{EMW}, yields
\[
 0=\frac{2(n-1)}{n-2}\int_\Sigma\partial_\nu\dot b_0\,d\sigma
   =\frac12\int_M\dot R_0\,dV_h
   =-\frac12\int_M
       \langle\operatorname{Ric}_h,\kappa\rangle_h\,dV_h,
\]
where dots denote derivatives at $t=0$. The divergence terms in
the scalar curvature variation integrate to zero because $\kappa$
has compact support. Since $\kappa$ is arbitrary,
$\operatorname{Ric}_h=0$. The boundary rigidity theorem
\cite[Lemma~5]{EMW} therefore applies: $\Sigma$ isometrically embeds
as the unit sphere in $\R^n$, and $H_h=n-1$ equals its Euclidean
mean curvature. Hence $(M,h)$ is isometric to the closed Euclidean
unit ball.
\end{proof}

\section{Two conformal identities}\label{sec:energy}

We now extract the two identities that drive the proof. The first
measures how much of the fixed exterior mass is removed by the global
conformal equation. The second controls the energy required to preserve
the given boundary while making the exterior exactly Euclidean.

\subsection{The mass identity}

Lemmas~\ref{lem:globalflatten} and~\ref{lem:corner} give
$m_\gamma\ge0$. On the fixed exterior,
\begin{equation}\label{eq:wexpansion}
 \gamma=w_{\infty}^{4/(n-2)}\delta,\qquad
 w_{\infty}=Fu=1+\frac{m_\gamma}{2r^{n-2}}+O(r^{1-n}),\qquad 0<w_{\infty}\le F.
\end{equation}
The smoothing used to prove the sign of the mass changes none of the
identities for the original metric.

\begin{lemma}\label{lem:budget}
The residual mass satisfies
\begin{equation}\label{eq:exactbudget}
 2(n-1)|\Sph^{n-1}| m_\gamma=\frac{C|\Sph^{n-1}|}{n}
       -\int_MR_gu\,dV_g
       -2\int_\Sigma(H_g-H_C)u\,d\sigma.
\end{equation}
In particular, with $J=\int_Mu\,dV_g$,
\[
 0\le2(n-1)|\Sph^{n-1}| m_\gamma
       \le C\left(\frac{|\Sph^{n-1}|}{n}-J\right),\qquad
 0\le m_\gamma\le\frac{C}{2n(n-1)}.
\]
These are the inequalities in \eqref{eq:mainbudget}.
\end{lemma}

\begin{proof}
Integrating the conformal equation converts its compact scalar and
boundary terms into the flux at infinity. That flux is exactly the
change from the fixed Schwarzschild mass to $m_\gamma$.
Integrate separately over $M$ and the exterior truncated by a large
coordinate sphere $S_r$. Use $\nu$ pointing from $M$ into the
exterior at $\Sigma$, and the outward normal on $S_r$. The
interior equation gives
\[
 \frac{4(n-1)}{n-2}
       \int_\Sigma\partial_{\nu,-}u\,d\sigma
          =\int_MR_gu\,dV_g.
\]
The exterior equation is $\Delta_{g_N}u=0$. Thus its flux through
$S_r$ equals the flux through $\Sigma$ measured in the direction
$\nu$. Substituting the transmission condition yields
\[
 \frac{4(n-1)}{n-2}
       \lim_{r\to\infty}
          \int_{S_r}\partial_{\nu_{g_N}}u\,dA_{g_N}
 =\int_MR_gu\,dV_g
       +2\int_\Sigma(H_g-H_C)u\,d\sigma.
\]

The fixed Schwarzschild mass is $C/[2n(n-1)]$. If
$u=1+b_{\infty}r^{2-n}+O(r^{1-n})$, the conformal mass formula and the
flux expansion give
\[
 m_\gamma=\frac{C}{2n(n-1)}+2b_{\infty},\qquad
 \lim_{r\to\infty}\int_{S_r}\partial_{\nu_{g_N}}u\,dA_{g_N}
       =-(n-2)|\Sph^{n-1}| b_{\infty}.
\]
The asymptotic metric corrections do not affect the flux limit because
$du=O(r^{1-n})$. Substitution proves \eqref{eq:exactbudget}.

The hypotheses $R_g\ge C$ and $H_g\ge H_C$, together with $u>0$,
make the two subtracted terms nonnegative and give
$\int_MR_gu\ge CJ$. The stated inequalities now follow from
$m_\gamma\ge0$.
\end{proof}

\subsection{Preserving the boundary and removing the exterior factor}

The Dirichlet conformal factor is obtained from the coercive compact
Dirichlet problem
\[
 -\frac{4(n-1)}{n-2}\Delta_gu_D+R_gu_D=0
       \quad\hbox{in }M,\qquad u_D|_\Sigma=1,
\]
and we set
\[
 h:=u_D^{4/(n-2)}g,\qquad V:=u_D^{-1}.
\]
Lax--Milgram and elliptic regularity give a unique smooth solution.
Here $h$ is initially defined only on $M$. The maximum principle gives $0<u_D\le1$. At the boundary, $u_D=1$ is its maximum,
so $\partial_{\nu_g}u_D\ge0$. The conformal curvature formulas
give
\begin{equation}\label{eq:dirichletdata}
 \begin{gathered}
 R_h=0,\qquad h_\Sigma=\sigma,\qquad V\ge1,\qquad V|_\Sigma=1,\\
 H_h=H_g+\frac{2(n-1)}{n-2}\partial_{\nu_g}u_D\ge H_C>0,\\
 -\frac{4(n-1)}{n-2}\Delta_hV
      =R_gV^{(n+2)/(n-2)}\ge CV^{(n+2)/(n-2)}.
 \end{gathered}
\end{equation}
The last equation follows by applying the scalar curvature formula
to $g=V^{4/(n-2)}h$. Scalar flatness of $h$ does not yet imply
that $h$ is Euclidean.

The two conformal factors $u$ and $u_D$ solve the same linear equation on $M$ but have different normalizations: $u$ is normalized at infinity on $N$, whereas $u_D=1$ on $\Sigma$. In particular, $u|_\Sigma$ is not
prescribed. Their exact relation on the fill-in is
\[
 \gamma|_M=u^{4/(n-2)}g,\qquad
 h=u_D^{4/(n-2)}g
   =\left(\frac{u_D}{u}\right)^{4/(n-2)}\gamma|_M.
\]
Conformal covariance gives
\[
 L_\gamma\left(\frac{u_D}{u}\right)
     =u^{-(n+2)/(n-2)}L_g u_D=0.
\]
Since $R_\gamma=0$, the ratio $u_D/u$ is $\gamma$-harmonic on $M$,
with boundary value $1/u|_\Sigma$. Thus $h$ is obtained from $\gamma$
by a specific harmonic conformal change, not by an unrelated choice
of scalar flat representative. The next construction multiplies this
ratio by the constant $z$ from \eqref{eq:parameters}, so that it
extends to the exterior factor $1/w_{\infty}$ and tends to one at
infinity. This normalization accounts for the constant rescaling in
\eqref{eq:uv} and \eqref{eq:hgamma}.

Apply Lemma~\ref{lem:BYpositive} to $(M,h)$ to obtain
\begin{equation}\label{eq:BY}
 \int_\Sigma(n-1-H_h)\,d\sigma\ge0.
\end{equation}
This is an integral inequality; no pointwise upper bound for $H_h$
is being asserted. 
We now construct the different exterior that will be used in the coordinate argument.

Let $v$ solve $\Delta_\gamma v=0$ in $M$ with boundary value
$1/w_{\infty}$, and define $v:=1/w_{\infty}$ on the exterior. The function is positive
and continuous on $N$. It is harmonic for $\gamma$ on each side:
on the exterior this follows from conformal covariance and
$\Delta_\delta w_{\infty}=0$. Its normal derivatives are not required to
agree across $\Sigma$.

Inside $M$, conformal covariance shows that $uv$ solves the same
linear equation as $u_D$. Its boundary value is constant, since
\[
 uv|_\Sigma=\frac{u}{w_{\infty}}\bigg|_\Sigma
          =\frac1{F(r_{\mathrm S})}=z.
\]
Dirichlet uniqueness therefore gives $uv=zu_D$. Consequently
\begin{equation}\label{eq:uv}
 uv=zu_D,\qquad
 v^{4/(n-2)}\gamma
 =\begin{cases}
      r_{\mathrm S}^2h,&\text{on }M,\\
      \delta,&\text{on the exterior}.
   \end{cases}
\end{equation}
The interior scaling is correct because
$z^{4/(n-2)}=r_{\mathrm S}^2$. On the exterior, the factors
$v^{4/(n-2)}$ and $w_{\infty}^{4/(n-2)}$ cancel exactly.

Since $w_{\infty}\le F$ and $F(r)\le F(r_{\mathrm S})=z^{-1}$ for
$r\ge r_{\mathrm S}$, one has $v\ge z$ on the exterior.
The minimum principle gives the same bound in $M$. Define the
Dirichlet energy of the conformal factor $v$ by
\[
 P:=\int_N|dv|_\gamma^2\,dV_\gamma.
\]
In particular, $0<v^{-1}\le z^{-1}$ and $v^{-1}=w_{\infty}$ on the exterior. The energy $P$
is finite: the compact pieces are smooth up to $\Sigma$, and
$dv=O(r^{1-n})$ on the end.

Extend the notation $h$ to all of $N$ by setting
\begin{equation}\label{eq:hgamma}
 h:=r_{\mathrm S}^{-2}v^{4/(n-2)}\gamma.
\end{equation}
This agrees with the already defined $h$ on $M$. Outside $M$ it
is $r_{\mathrm S}^{-2}\delta_x$, hence exactly the Euclidean
metric $\delta_y$ in the coordinates $y:=x/r_{\mathrm S}$.
The identified boundary is now $|y|=1$. The continuous metric $h$
need not have nonnegative distributional scalar curvature across
$\Sigma$; positive mass and minimal graphs will be applied to
$\gamma$ and its smooth approximations, not to this extension of
$h$.

The inverse factor is harmonic for $h$ on $M$: conformal covariance
gives $L_{v^{4/(n-2)}\gamma}(v^{-1})=v^{-(n+2)/(n-2)}L_\gamma1=0$, and
constant rescaling preserves harmonicity.

\subsection{The energy identity}

The mean curvature jump of the extension of $h$ records the energy of
$v$. The Brown--York sign in \eqref{eq:BY} therefore converts the
identity below into a mass controlled energy estimate. The two boundary
estimates for $v^{-1}$ will later control its sign-changing deviation
from one.

\begin{lemma}\label{lem:energy}
The following identities and estimates hold:
\begin{align}
 z^2\int_\Sigma(n-1-H_h)\,d\sigma
    &=(n-1)|\Sph^{n-1}| m_\gamma-\frac{2(n-1)}{n-2}P,
                                      \label{eq:energyidentity}\\
 P&\le\frac{(n-2)|\Sph^{n-1}|}{2} m_\gamma,
                                      \label{eq:energybound}\\
 \int_N|d(v^{-1})|_h^2\,dV_h&\le z^{-4}P,\qquad
 \int_\Sigma(v^{-1}-1)^2\,d\sigma
       \le\frac{|\Sph^{n-1}| m_\gamma}{2z^4},
                                      \label{eq:inverse-v-energy}\\
 \int_\Sigma(v^{-1}-1)\,d\sigma
    &=\frac{|\Sph^{n-1}| m_\gamma}{2z^2}\ge0,
                                      \label{eq:inverse-v-mean}\\
 J&=z\int_Mv^{-1}V^{(n+2)/(n-2)}\,dV_h.
                                      \label{eq:Jweighted}
\end{align}
\end{lemma}

\begin{proof}
First compare the two mean curvatures of the continuous extension
$h=r_{\mathrm S}^{-2}v^{4/(n-2)}\gamma$. Both one-sided
$\gamma$-normals point from $M$ into the exterior. The mean curvatures
of $\gamma$ agree, while those of $h$ are $H_h$ on the compact side
and $n-1$ on the exterior. The conformal mean curvature and area
formulas consequently yield
\[
 z^2\int_\Sigma(n-1-H_h)\,d\sigma
 =\frac{2(n-1)}{n-2}\int_\Sigma
      v(\partial_{\nu,+}v-\partial_{\nu,-}v)\,dA_\gamma.
\]
The scale factor is $z^2=r_{\mathrm S}^{n-2}$: mean curvature
scales by $r_{\mathrm S}$ and area by $r_{\mathrm S}^{1-n}$.
The remaining conformal powers combine to the factor $v$.

Integration of $v\Delta_\gamma v=0$ over the two pieces yields
\[
 P=\lim_{r\to\infty}
        \int_{S_r}v\partial_{\nu_\gamma}v\,dA_\gamma
      -\int_\Sigma v(\partial_{\nu,+}v-\partial_{\nu,-}v)
                                                       \,dA_\gamma.
\]
The sign follows from the inward-pointing outward normal of the
exterior region at its inner boundary. Since $v=w_{\infty}^{-1}$,
\eqref{eq:wexpansion} gives
\[
 v=1-\frac{m_\gamma}{2r^{n-2}}+O(r^{1-n}),\qquad
 \lim_{r\to\infty}
      \int_{S_r}v\partial_{\nu_\gamma}v\,dA_\gamma
         =\frac{(n-2)|\Sph^{n-1}| m_\gamma}{2}.
\]
Substitution proves \eqref{eq:energyidentity}. The nonnegativity
of its left side, by \eqref{eq:BY}, proves
\eqref{eq:energybound}.

For the first estimate in \eqref{eq:inverse-v-energy}, the conformal and
constant scalings in \eqref{eq:hgamma} give
\[
 \int_N|d(v^{-1})|_h^2\,dV_h
   =r_{\mathrm S}^{2-n}\int_Nv^2|d(v^{-1})|_\gamma^2\,dV_\gamma
   =z^{-2}\int_Nv^{-2}|dv|_\gamma^2\,dV_\gamma
   \le z^{-4}P.
\]
Here $d(v^{-1})=-v^{-2}dv$ and $v^{-1}\le z^{-1}$.

On the exterior, cancellation of the conformal powers gives
\[
 \int_{r\ge r_{\mathrm S}}|dv|_\gamma^2\,dV_\gamma
       =\int_{r\ge r_{\mathrm S}}|D\log w_{\infty}|^2\,dx.
\]
Using $w_{\infty}\le z^{-1}$, we obtain
$\int_{r\ge r_{\mathrm S}}|Dw_{\infty}|^2\,dx\le z^{-2}P$.
For each $\theta\in\Sph^{n-1}$, radial integration and
Cauchy--Schwarz give
\begin{align*}
 |w_{\infty}(r_{\mathrm S}\theta)-1|^2
 &\le\left(\int_{r_{\mathrm S}}^\infty r^{1-n}\,dr\right)
       \left(\int_{r_{\mathrm S}}^\infty
                       |\partial_rw_{\infty}(r\theta)|^2r^{n-1}\,dr\right)\\
 &=\frac{r_{\mathrm S}^{2-n}}{n-2}
       \int_{r_{\mathrm S}}^\infty
                       |\partial_rw_{\infty}(r\theta)|^2r^{n-1}\,dr.
\end{align*}
The prescribed boundary measure is $d\sigma(\theta)$, not the
Euclidean area measure of the coordinate sphere of radius
$r_{\mathrm S}$. Integrating in $\theta$ and using
$r_{\mathrm S}^{2-n}=z^{-2}$ yields
\[
 \int_\Sigma(v^{-1}-1)^2\,d\sigma
       \le\frac{P}{(n-2)z^4}
       \le\frac{|\Sph^{n-1}| m_\gamma}{2z^4},
\]
which proves the second estimate in \eqref{eq:inverse-v-energy}.

The spherical average of the exterior harmonic function $w_{\infty}$
satisfies the radial harmonic equation and has limit one at infinity.
Its coefficient of $r^{2-n}$ is fixed by \eqref{eq:wexpansion}; hence
\[
 \frac{1}{|\Sph^{n-1}|}\int_{\Sph^{n-1}}w_{\infty}(r\theta)\,d\sigma
       =1+\frac{m_\gamma}{2r^{n-2}}
       \qquad(r\ge r_{\mathrm S}).
\]
Evaluating at $r=r_{\mathrm S}$, where $v^{-1}=w_{\infty}$, proves
\eqref{eq:inverse-v-mean}. This nonnegative boundary average is stronger
than merely an $L^2$ bound and will be used in the weighted
comparison.

Finally, \eqref{eq:uv} gives $u=zv^{-1}u_D=zv^{-1}/V$, whereas
$g=V^{4/(n-2)}h$ gives
$dV_g=V^{2n/(n-2)}dV_h$. Therefore
\[
 J=\int_Mu\,dV_g
   =z\int_Mv^{-1}V^{2n/(n-2)-1}\,dV_h
   =z\int_Mv^{-1}V^{(n+2)/(n-2)}\,dV_h,
\]
which is \eqref{eq:Jweighted}.
\end{proof}

Thus, by \eqref{eq:energybound}, the conformal factor \(v/z\) relating the mass metric  \(\gamma\) to the boundary preserving metric \(h\) has Dirichlet energy, measured with respect to \(\gamma\), bounded by \(m_\gamma\). We use it first to locate the boundary in quantitative coordinates and then to estimate the weighted
integral in \eqref{eq:Jweighted}.

\section{Coordinates and the prescribed boundary}\label{sec:geometry}

The graph construction gives a map with small metric distortion on a
large sublevel set. To use the prescribed boundary, three additional
facts are needed: a single inverse branch containing the end, a
parametrization of the image of the round sphere, and identification of
the target interior with points of the fill-in. The last fact
does not follow merely from locating the boundary. We prove these
statements after removing sets of small capacity; their topology and
boundary regularity will play no role.

Throughout this section, $h=r_{\mathrm S}^{-2}v^{4/(n-2)}\gamma$
is exactly Euclidean on the attached exterior
$X=\{|y|\ge1\}$, where $y:=x/r_{\mathrm S}$ and $x$ is the
asymptotic coordinate used above. Three spaces must be distinguished:
the manifold $N=M\cup_\Sigma X$; a reference copy of $\R^n$ in which
$X$ is the complement of the unit ball $B$; and the target copy of
$\R^n$, with coordinate $\xi=\bar\Phi(a)$ for $a\in N$.
Only the exteriors of the first two spaces are identified. In
particular, the reference ball $B$ is not a coordinate model for $M$.
The map $\bar\Phi$ is defined globally, but its inverse will be used
only on a specified open set. The global Euclidean map $F_0$ constructed
below extends an exterior displacement into $B$; it does not extend a
coordinate chart through the fill-in.

For a compact set $\mathcal K\subset\R^n$, write
\[
 \Cap(\mathcal K):=\inf\left\{\int_{\R^n}|Da|^2:
      a\in C_c^\infty(\R^n),\ a\ge1\text{ near }\mathcal K\right\}.
\]
This is the unnormalized Newtonian capacity. The capacity $\Cap_h$ on $(N,h)$ is defined
by the same formula, using compactly supported functions on $N$ and the Dirichlet energy
of $h$.

The main conclusion of the section is the following proposition.
It records the domain, the boundary parametrization, and all the estimates
needed in the comparison argument.

\begin{proposition}\label{prop:marked}
Fix a sufficiently small dimensional amplitude $\tau>0$.
For all sufficiently small $C$ and arbitrarily small
$\varepsilon>0$, there are a radial domain
\[
 D_f:=\{r\theta:0\le r<1+f(\theta)\},
 \qquad f\in W^{1,\infty}(\Sph^{n-1}),
\]
a compact set $A\subset\overline{D_f}$, and an
orientation-preserving coordinate map $\bar{\Phi}$ from an open subset of
$\operatorname{int}M$ onto $O:=D_f\setminus A$.
In these coordinates, write
$q:=(\bar\Phi^{-1})^*h-\delta$; thus $R_{\delta+q}=0$ on $O$.
There are a symmetric tensor
$q_{\mathrm{ref}}\in H^1(D_f)$, equal to $q$ on $O$, and a global bi-Lipschitz
map $F_0:\R^n\to\R^n$ with $F_0(B)=D_f$, such that
\[
 \max\{\|q\|_\infty,\|q_{\mathrm{ref}}\|_\infty,
                   \|f\|_{W^{1,\infty}}\}\le\tau,
 \qquad \operatorname{Lip}(F_0),\operatorname{Lip}(F_0^{-1})\le2.
\]
Let $\Gamma:=\partial D_f\setminus A$ be the part of the outer boundary outside $A$.
Under the prescribed identification
$\Sigma=\Sph^{n-1}$, the inverse coordinate map extends to
$\Gamma$ and agrees there with $F_0^{-1}$. Thus $F_0^{-1}(\Gamma)\subset\Sigma$ is the
corresponding part of the original boundary.

For $\mathscr E:=K(n)(m_\gamma+\varepsilon)$, where the dimensional constant
is chosen after fixing $\tau$, the following quantities have sum
at most $\mathscr E$:
\begin{align*}
 &\|q\|_{H^1(O)}^2+\|q_{\mathrm{ref}}\|_{H^1(D_f)}^2
   +\|f\|_{H^1(\Sph^{n-1})}^2
   +\|F_0-\operatorname{Id}\|_{W^{1,2}(\Sph^{n-1})}^2,\\
 &\Cap(A)+|A|
   +|\Sigma\setminus F_0^{-1}(\Gamma)|_\sigma
   +|\partial D_f\setminus\Gamma|_\delta,\\
 &\|v^{-1}-1\|_{H^1(O)}^2
   +\|v^{-1}-1\|_{L^2(\Gamma)}^2.
\end{align*}
In the last line, $v^{-1}-1$ is written in the coordinates $\bar{\Phi}$.
In particular, the area $\operatorname{Area}_h(\Gamma)$ differs from
$|\Sph^{n-1}|$ by at most $K\mathscr E$. More precisely,
\[
 |A|\le K\mathscr E^{n/(n-2)},\qquad
 |\Sigma\setminus F_0^{-1}(\Gamma)|_\sigma
   +|\partial D_f\setminus\Gamma|_\delta
       \le K\mathscr E^{(n-1)/(n-2)}.
\]
The same bound holds for $\bigl|\operatorname{Area}_h(\Gamma)-|\Sph^{n-1}|\bigr|$. Here the
subscripts $\sigma$ and $\delta$ indicate boundary area in the
round source sphere and in Euclidean target coordinates, respectively.
\end{proposition}

We prove the proposition through four lemmas. They provide the local
coordinate estimates, injectivity on an open set containing the end,
the boundary parametrization, and the identification of the part of $M$
corresponding to the interior of $D_f$. The proof of the proposition then assembles these results
and extends the metric perturbation across the omitted set.

\begin{figure}[tbp]
\centering
\resizebox{\textwidth}{!}{\begin{tikzpicture}[x=1cm,y=1cm,>=Latex,font=\small]
\definecolor{selgreen}{RGB}{34,125,101}
\definecolor{badred}{RGB}{174,57,56}
\definecolor{refblue}{RGB}{47,87,140}
\definecolor{holepurple}{RGB}{125,75,143}
\node[anchor=west,font=\bfseries] at (-2.1,2.7) {Glued manifold};
\node[anchor=west,font=\bfseries] at (3.25,2.7) {Reference space};
\node[anchor=west,font=\bfseries] at (8.6,2.7) {Target space};
\begin{scope}
\fill[selgreen!10] (-2.1,-2.0) rectangle (2.1,2.1);
\draw[fill=white,thick] (0,0) circle (1.13);
\fill[selgreen!12] (-.96,.5) .. controls (-.25,.9) and (.6,.95) .. (.98,.52) -- (1.1,.8) arc (36:155:1.36) -- cycle;
\draw[selgreen,dashed,thick] (-.96,.5) .. controls (-.25,.9) and (.6,.95) .. (.98,.52);
\draw[fill=badred!30,draw=badred] (1.45,.45) ellipse (.27 and .47);
\draw[fill=badred!30,draw=badred] (-.6,.02) ellipse (.22 and .17);
\node at (0,-.48) {$(M,h)$};
\node[selgreen] at (-1.45,1.65) {$E$};
\node at (1.6,-1.55) {$X$};
\node[anchor=east] at (-1.3,-1.3) {$\Sigma$};
\draw (-1.3,-1.16)--(-.76,-.84);
\node[align=center,text width=4.2cm] at (0,-2.7) {$\bar\Phi:N\to\mathbb R^n$ is global.\\Its selected inverse is defined only on $Y$.};
\end{scope}
\begin{scope}[xshift=5.35cm]
\fill[refblue!5] (-2.1,-2.0) rectangle (2.1,2.1);
\draw[fill=white,thick] (0,0) circle (1.13);
\fill[orange!25] (22:1.13) arc (22:67:1.13) -- (67:1.83) arc (67:22:1.83) -- cycle;
\draw[orange!70!black] (22:1.13) arc (22:67:1.13) -- (67:1.83) arc (67:22:1.83) -- cycle;
\fill[badred!65] (33:1.13) arc (33:54:1.13) -- (54:1.51) arc (54:33:1.51) -- cycle;
\node at (-.15,-.1) {$B$};
\node[refblue] at (-1.3,1.5) {$X$};
\node[anchor=west] at (1.35,1.7) {$A_X$};
\draw (1.38,1.56)--(1.14,1.24);
\node[anchor=west] at (1.3,.5) {$K_X$};
\draw (1.3,.51)--(1.02,.94);
\node[anchor=east] at (-1.2,-1.2) {$\Sph^{n-1}$};
\draw (-1.25,-1.1)--(-.77,-.83);
\node[align=center,text width=4.3cm] at (0,-2.7) {$K_X\subset A_X\subset X$.\\Reflection extends $a_X$ into $B$, not into $M$.};
\end{scope}
\begin{scope}[xshift=10.7cm]
\fill[refblue!5] (-2.1,-2.0) rectangle (2.1,2.1);
\path[draw=refblue,thick,fill=white] plot[smooth cycle] coordinates {(-1.15,0) (-.7,-.94) (.3,-1.08) (1.19,-.43) (1.13,.55) (.4,1.22) (-.7,.98)};
\fill[holepurple!45] (-.35,.2) ellipse (.23 and .18);
\draw[holepurple,dashed,thick] (-.35,.2) ellipse (.37 and .31);
\fill[orange!25] (1.0,.4) .. controls (1.45,.4) and (1.7,1.0) .. (1.2,1.6) -- (.8,1.2) .. controls (.65,.8) and (.7,.6) .. (1.0,.4);
\draw[orange!70!black] (1.0,.4) .. controls (1.45,.4) and (1.7,1.0) .. (1.2,1.6) -- (.8,1.2) .. controls (.65,.8) and (.7,.6) .. (1.0,.4);
\draw[holepurple,dashed,thick] (.4,.35) .. controls (.7,.15) and (1.48,.17) .. (1.69,.76) .. controls (1.98,1.55) and (1.5,1.9) .. (.9,1.77) .. controls (.7,1.22) and (.34,.77) .. (.4,.35);
\node at (-.18,-.57) {$D_f\setminus L$};
\node[holepurple,anchor=east] at (-.8,.53) {$K_0$};
\draw[holepurple] (-.75,.4)--(-.55,.28);
\node[font=\footnotesize,anchor=west] at (.45,2.05) {$F_0(A_X)\subset L$};
\node[refblue,anchor=east] at (-1.1,-1.45) {$F_0(\Sph^{n-1})$};
\draw[refblue] (-.98,-1.33)--(-.7,-.94);
\node[align=center,text width=4.3cm] at (0,-2.7) {$W_0=K_0\cup F_0(A_X)$.\\The set $L$ also contains the exceptional target potential level.};
\end{scope}
\end{tikzpicture}}
\caption{Three spaces and the exceptional sets (schematic two-dimensional sections, not a claim about their topology). Green indicates part of the selected chart $E$, which can cross $\Sigma$. The drawing of $M$ carries no assertion that $M$ is a ball. The reference ball $B$ is an auxiliary Euclidean object. On the reference exterior only, $\bar\Phi=F_0$ outside $A_X$. In the target, $\bar\Phi:E\to Y=\mathbb R^n\setminus K_0$ is one-to-one, while $F_0$ is globally one-to-one. After the larger exception $L$ is removed, $D_f\setminus L$ is proved to be the image of points of the interior of $M$.}
\label{fig:sets-atlas}
\vspace{4mm}
\begin{minipage}{.95\textwidth}
\small
The identities that carry information between the panels are
\[
\begin{gathered}
 \bar\Phi|_{X\setminus A_X}=F_0|_{X\setminus A_X},\qquad D_f=F_0(B),\\
 (\bar\Phi|_E)^{-1}(D_f\setminus L)\subset\operatorname{int}M.
\end{gathered}
\]
There is no asserted equality $E=X\setminus \bar\Phi^{-1}(K_0)$, no asserted global inverse of $\bar\Phi$, and no identification $F_0|_{B}=\bar\Phi|_M$. The final hole $A$ is obtained from a capacity cutoff around $L$; the comparison domain is $O=D_f\setminus A$.
\end{minipage}
\end{figure}
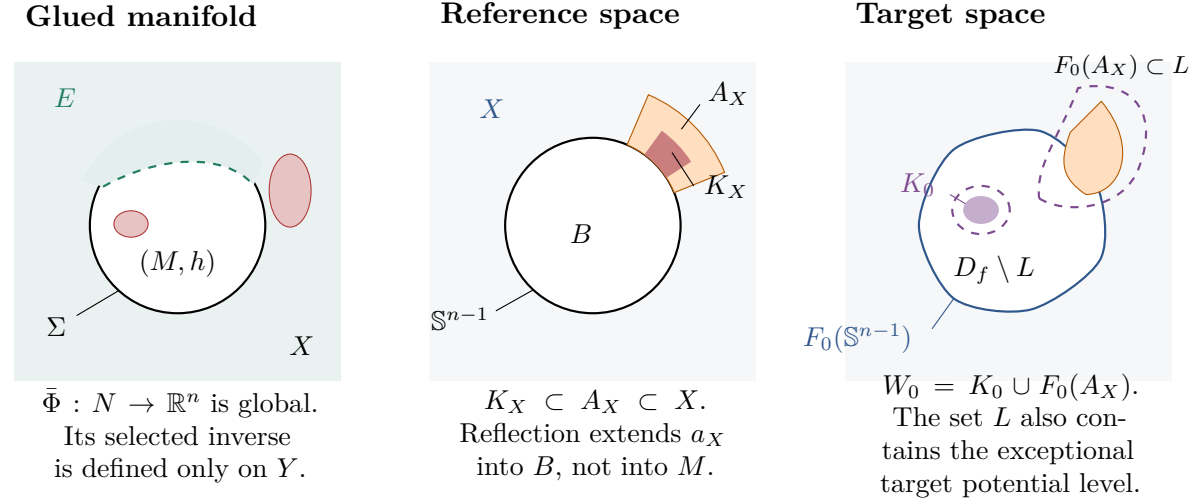

\subsection{Estimates for the coordinate map}

We first specify the coordinate error in Wang--Zhang's construction.
Let $(N,g')$ be smooth, connected, complete, oriented and without boundary,
with exactly one end that is harmonically flat outside a compact set:
$g'=\phi^{4/(n-2)}\delta$, $\Delta_\delta\phi=0$, $\phi>0$, and $\phi\to1$.
Suppose $R_{g'}\ge0$ and write $m'$ for its mass. These are the
properties of the smooth repaired metrics used below. Fix the dimensional
slope $\lambda=\lambda_0(n)/2>0$ from \cite[Theorem~3.1]{WZ},
as specified in Section~\ref{sec:foundations}.
The normalized minimal graph map
$\Phi=(\Phi_1,\ldots,\Phi_n)$ is chosen so that
$\lambda\Phi_i$ is the minimal graph with asymptotic slope
$\lambda e_i$. These functions solve
\[
 \operatorname{div}_{g'}\left(
   \frac{\nabla_{g'}\Phi_i}
        {\sqrt{1+\lambda^2|d\Phi_i|_{g'}^2}}\right)=0.
\]
For each component, the conformally modified graph metric and its
scalar potential are
\[
 \widetilde g_i
   :=(1+\lambda^2|d\Phi_i|_{g'}^2)^{-1/(n-1)}
       (g'+\lambda^2d\Phi_i\otimes d\Phi_i),
\]
\[
 -\frac{4(n-1)}{n-2}\Delta_{\widetilde g_i}\varphi_i
       +R_{\widetilde g_i}\varphi_i=0,
 \qquad \varphi_i\longrightarrow1\text{ on the end}.
\]
These potentials satisfy $0<\varphi_i\le1$ by
\cite[Lemma~6.1]{WZ}. They are auxiliary functions associated with
the modified graph metrics, distinct from the exterior harmonic
function $w_{\infty}=Fu$ used earlier. Define the nonnegative function $Q$ by
\begin{equation}\label{eq:marked-Q-definition}
 Q:=\sum_{i,j=1}^n
       \bigl(\langle d\Phi_i,d\Phi_j\rangle_{g'}-\delta_{ij}\bigr)^2
       +\sum_{i=1}^n(1-\varphi_i^2).
\end{equation}
Writing $\mathsf a_{ij}:=\langle d\Phi_i,d\Phi_j\rangle_{g'}$
and letting $I$ denote the identity matrix,
\cite[Theorem~3.1, Proposition~7.1 and equation~(7.7)]{WZ} gives
the estimates
\begin{gather}
 \Phi(x)=x+o(|x|),\qquad D_x\Phi=I+o(1),\qquad Q\longrightarrow0,
 \qquad Q\ge|\mathsf a-I|^2,\label{eq:marked-WZ-map}\\
 \int_{\{Q<6\beta\}}\sum_i|\Hess_{g'}\Phi_i|^2\dd V_{g'}\le Km',
 \qquad
 \int_{\{Q<6\beta\}}|dQ|_{g'}^2\dd V_{g'}\le K\beta m'.
                                      \label{eq:marked-WZ-energy}
\end{gather}
Here $0<\beta<\beta_0(n)$, where $\beta_0$ is the dimensional
threshold in \cite[Proposition~7.1]{WZ}. The fixed slope has been
absorbed into $K$. The function $Q$ is exactly their combined defect
\cite[equation~(7.1)]{WZ}. Existence and global gradient bounds on
the compact core may depend on $g'$, but the cited local estimates
have dimensional constants.
Their Section~7 treats $m'\ge0$. At $m'=0$, the same bounds follow
from the nonnegative mass--energy identity
\cite[equation~(6.7)]{WZ}: the potential gradients and weighted scalar
terms vanish, and the scalar density identity
\cite[equation~(7.6)]{WZ} forces the local Hessian energy to vanish.
Differentiating the Gram matrix then gives $dQ=0$ on the sublevel set.
This deduction uses no division by the mass or mass rigidity theorem. The covariant Hessians and norms are
computed using $g'$, whereas $D_x\Phi$ denotes ordinary
differentiation in the asymptotic chart. The Gram estimate controls
metric distortion, and the Hessian estimate controls variation of
the differential on the same sublevel set. These are estimates from
the construction, rather than conclusions drawn from metric
convergence alone.
The map is proper: a sequence leaving every compact subset of the
one-ended manifold lies eventually in the asymptotic chart, where
$|\Phi(x)|\to\infty$ by \eqref{eq:marked-WZ-map}. Properness will
prevent inverse branches from disappearing inside the target.

Apply this input to a smooth approximation with nonnegative scalar curvature
$\gamma_\varepsilon$ of $\gamma$ supplied by
Lemma~\ref{lem:corner}. The end from Lemma~\ref{lem:globalflatten}
remains exactly harmonically flat under that repair, since
$\Delta_\delta(w_{\infty}b_\varepsilon)=0$ there. Thus it satisfies the
specialized graph hypotheses. Choose its mass error, uniform metric error, and squared $L^2$
connection error smaller than $\varepsilon$. If needed, decrease the
smoothing width further so that all fixed fill-in transfer errors below
are smaller than the same prescribed error.
Write $\Phi$ for its graph map and use
\[
 \bar{\Phi}:=\Phi/r_{\mathrm S}.
\]
The Gram matrix computed using $h$ is exactly
\[
 (\langle d\bar{\Phi}_i,d\bar{\Phi}_j\rangle_h)_{ij}
       =v^{-4/(n-2)}(\langle d\Phi_i,d\Phi_j\rangle_\gamma)_{ij}.
\]
This identity follows directly from conformal scaling of covector
inner products and the normalization $\bar\Phi=\Phi/r_{\mathrm S}$.
A small value of $Q$ alone is insufficient:
$Q$ measures the differential using $\gamma_\varepsilon$, whereas
the coordinates must be controlled for $h$. We therefore include
the conformal factor when choosing the set to remove. Smooth $v$ in a compact collar of
$\Sigma$, obtaining
$v_\varepsilon$ with uniform error tending to zero and
\[
 \int_N|dv_\varepsilon|_\gamma^2\dd V_\gamma\le P+\varepsilon.
\]
It is unchanged off the collar; no global $L^2$ bound for $v-1$ is
being asserted. Define the smooth function
\[
 S:=Q+(v_\varepsilon-1)^2.
\]
\begin{lemma}\label{lem:marked-joint}
For fixed sufficiently small $\beta$ and sufficiently accurate
approximations, $\bar{\Phi}$ is a local diffeomorphism on
$\{S<6\beta\}$. The singular values of $d\bar{\Phi}$, measured
from $(N,h)$ to Euclidean space, lie between two fixed positive
constants, and its Gram matrix differs from $I$ by
$K\sqrt\beta+o(1)$. The function $S$ satisfies
\begin{equation}\label{eq:marked-joint-energy}
 \begin{gathered}
 \int_{\{S<6\beta\}}\sum_i|\Hess_h\bar{\Phi}_i|^2\dd V_h\le K(m_\gamma+\varepsilon),
 \\
 \int_{\{S<6\beta\}}|dS|_h^2\dd V_h\le K\beta (m_\gamma+\varepsilon).
 \end{gathered}
\end{equation}
\end{lemma}
\begin{proof}
Both terms in $S$ are nonnegative and tend to zero at infinity, so
its small sublevel sets contain the end. On $\{S<6\beta\}$,
the Gram identity and uniform metric approximation control the
differential of $\bar{\Phi}$. The conformal connection formula,
\eqref{eq:marked-WZ-energy}, and \eqref{eq:energybound} give the
energy estimates. On this band $|d\Phi|$ is bounded and
\[
 |\Hess_h\bar{\Phi}|\le K\bigl(|\Hess_{\gamma_\varepsilon}\Phi|
       +|\nabla^{\gamma_\varepsilon}-\nabla^\gamma|
       +|dv|_\gamma\bigr).
\]
Moreover $|d((v_\varepsilon-1)^2)|^2\le24\beta|dv_\varepsilon|^2$,
and the scalar differential $dQ$ requires only uniform metric
comparison. The bounded scale factors involving $r_{\mathrm S}$ enter $K$.
Here $\Hess_h\bar{\Phi}$ is the covariant Hessian for the metric
$h$; the connection difference compares the two connections acting on
the same covector $d\Phi$. All norms in this comparison are uniformly
equivalent on the indicated sublevel set. The metric $h$ is
continuous across $\Sigma$, smooth on either side, and has locally
bounded first derivatives. The displayed estimates are consequently
interpreted almost everywhere. They use first derivatives of $v$
and of the continuous metric $h$. Although these derivatives can
have different one-sided traces, they are ordinary almost everywhere
functions. No distributional second derivative of $v$ or $h$ is
taken in this argument.

\end{proof}

\subsection{Capacity and injectivity}

Capacity is useful here because it controls both the size of an
omitted set and the Dirichlet energy needed to cut off a function near
that set. The latter control will be needed when extending the metric
and the harmonic function across the omitted region. We use the
capacities defined before Proposition~\ref{prop:marked}; standard
background on their Sobolev formulation is given in
\cite{AdamsHedberg}. In Euclidean space, the Sobolev inequality and
its trace on the unit sphere imply
\begin{equation}\label{eq:marked-capacity-size}
 |\mathcal K|^{(n-2)/n}
   +|\mathcal K\cap\Sph^{n-1}|^{(n-2)/(n-1)}\le K(n)\Cap(\mathcal K).
\end{equation}
To see this, apply the inequalities to any admissible function and use
that its value is at least one on the indicated sets. The same
argument bounds the $H^1(B_2)$ and spherical $L^2$ norms of a capacity
cutoff by its Dirichlet norm. A map and its inverse with Lipschitz
constants at most two distort capacity by a dimensional factor:
compose a test function with the inverse map and change variables.
Fill-in the bounded complementary components of a compact set does
not increase capacity. Replace an admissible test by its truncation
to $[0,1]$, then by one in those components; this deletes energy and
preserves the trace across their boundaries. Approximation, or the
equivalent Sobolev definition of capacity, gives the assertion.
Cutoffs below can always be chosen with arbitrarily small excess
energy, absorbed into $\varepsilon$.

\begin{lemma}\label{lem:marked-sheet}
For the map and function $S$ in Lemma~\ref{lem:marked-joint}, there are
a compact set $K_0\subset\R^n$ and an open set
$E\subset N$ with compact complement such that
$Y:=\R^n\setminus K_0$ is connected and
\begin{equation}\label{eq:marked-sheet}
 \bar{\Phi}:E\longrightarrow Y
\end{equation}
is an orientation-preserving diffeomorphism. Moreover,
\begin{equation}\label{eq:marked-two-capacities}
 \Cap_h(N\setminus E)+\Cap(K_0)\le K(m_\gamma+\varepsilon).
\end{equation}
In coordinates on $Y$, set
$q:=(\bar{\Phi}^{-1})^*h-\delta$. Then
\begin{equation}\label{eq:marked-metric-estimate}
 \begin{gathered}
 \|q\|_\infty\le K\sqrt\beta+o(1),\qquad
 \int_{Y}|Dq|^2\le K(m_\gamma+\varepsilon),
 \\
 \|q\|_{L^{2n/(n-2)}(Y)}^2\le K(m_\gamma+\varepsilon).
 \end{gathered}
\end{equation}
The error $o(1)$ tends to zero with the approximation errors for
the fixed fill-in.
\end{lemma}
\begin{proof}
We first transfer the Dirichlet energy of $S$ to the target without
assuming that $\bar{\Phi}$ is globally injective. Choose a smooth scalar
function $\chi$ supported in $(\beta,4\beta)$,
equal to one near $[2\beta,3\beta]$. Since $\bar{\Phi}$ is a local
diffeomorphism on this band, the target function
\[
 b_*(\xi):=\max\{\chi(S(a)):\bar{\Phi}(a)=\xi\},
\]
with value zero if the fibre contributes no point, is locally
Lipschitz and compactly supported. In fact the contributing source
set is compact, and its fibres are discrete, hence finite. Near a
fixed target point, every nonzero contribution therefore comes from
one of finitely many inverse branches: otherwise a sequence of
additional contributing source points would have a convergent
subsequence giving another preimage there. Each branch is cut off
before leaving the local diffeomorphism region. The area formula
and \eqref{eq:marked-joint-energy} give
\[
 \int_{\R^n}|Db_*|^2\le
 K\int_{\{S<4\beta\}}|dS|_h^2\dd V_h\le K(m_\gamma+\varepsilon).
\]
At almost every target point, the derivative of the maximum is the
derivative of an active branch. If branches agree on a measurable
set, their weak derivatives agree almost everywhere on that set;
this follows by applying locality to their difference. Thus ties do
not affect the estimate. The squared norm is bounded by the
sum of the squared norms over all branches. Changing variables on
each branch gives the displayed estimate, since both its differential
and its inverse have bounded distortion. Thus multiplicity causes
no loss, and no global injectivity has been used at this stage.

Choose a regular value $s_0\in(2\beta,3\beta)$ and the component
$\mathcal C$ of $\{S<s_0\}$ containing the end. Its boundary is
compact, and $b_*=1$ on a neighborhood of $\bar{\Phi}(\partial\mathcal C)$.
Choose a compact neighborhood of that image with capacity at most
$K(m_\gamma+\varepsilon)$, fill its bounded complementary components, and call the
result $K_0$. Then
\[
 \Cap(K_0)\le K(m_\gamma+\varepsilon),\qquad Y=\R^n\setminus K_0
\]
is connected and belongs to the unbounded component of
$\R^n\setminus \bar{\Phi}(\partial\mathcal C)$.

The Jacobian of $\bar{\Phi}$ is positive throughout $\mathcal C$: it never
vanishes there and is positive on the end. Over $Y$ the number of
preimages in $\mathcal C$ is finite and locally constant. Properness
prevents escape, and crossing $\partial\mathcal C$ would put the
image in $K_0$. That number is one near infinity. Indeed, a large
end sphere has degree one under $\bar{\Phi}$, by its straight homotopy to the
identity; sufficiently distant target points have all their preimages
on the end, with positive local degree. Therefore the restriction in \eqref{eq:marked-sheet} is an
orientation-preserving diffeomorphism, where
$E:=\mathcal C\cap \bar{\Phi}^{-1}(Y)$. The set $\mathcal C$ may meet
both sides of $\Sigma$; it is the component containing the end, not
a subset of the attached exterior. From this point onward,
$\bar{\Phi}^{-1}$ always
means this inverse, even if the original map has other preimages
outside $\mathcal C$. The complement of $E$ in $N$ is
compact, because $E$ contains all sufficiently distant points of the end.

The target capacity estimate also controls $\Cap_h(N\setminus E)$.
By the definition of $E$ and the inclusion
$\bar{\Phi}(\partial\mathcal C)\subset K_0$, every frontier point of
$E$ maps into $K_0$. Pull a target cutoff equal to one near
$K_0$ back to $E$ and extend it by one on $N\setminus E$.
The pullback is locally constant across this frontier, so the extension
creates no derivative jump. Bounded distortion gives
\eqref{eq:marked-two-capacities}.
These are two capacity estimates, in two different ambient spaces.
They need not represent disjoint omitted regions. No Sobolev
inequality on $M$ is used.

On $Y$ write the metric $h$ in coordinates as $(\bar{\Phi}^{-1})^*h=\delta+q$. The
inverse of this metric matrix
is the Gram matrix of the covectors $d\bar{\Phi}_i$. Its derivative is
controlled by $\Hess_h\bar{\Phi}$, because differentiating an inner product
using the metric connection introduces precisely these covariant
Hessians. The uniform bounds on the Gram matrix permit inversion,
and changing variables through the inverse gives the
first two estimates in \eqref{eq:marked-metric-estimate}.
For the last estimate, let $a$ be a compact capacity cutoff equal
to one near $K_0$. Extend $(1-a)q$ by zero through $K_0$.
Each component is locally $H^1$, has finite Dirichlet energy, is bounded,
and tends to zero at infinity. If $b$ is such a component, then for
$t>0$ its value truncation
\[
 T_t(b):=\operatorname{sgn}(b)(|b|-t)_+
\]
has compact support and $|DT_t(b)|\le|Db|$ almost everywhere.
Euclidean Sobolev followed by monotone convergence as $t\downarrow0$
gives $\|b\|_{2n/(n-2)}^2\le K\|Db\|_2^2$.
The product rule therefore yields
\[
 \|(1-a)q\|_{2n/(n-2)}^2
 \le K\int_{Y}|Dq|^2
       +K\|q\|_\infty^2\int|Da|^2.
\]
The remaining product is bounded by
$\|aq\|_{2n/(n-2)}\le\|q\|_\infty\|a\|_{2n/(n-2)}$;
apply Sobolev to $a$ and send the excess cutoff energy to zero.
This proves the final estimate in \eqref{eq:marked-metric-estimate}
with a dimensional constant. In particular, the squared $H^1$ norm
on any target set of bounded volume is at most $K(m_\gamma+\varepsilon)$.
Continuity of the metric coefficients across the image of $\Sigma$
ensures that their first weak derivatives introduce no interface
measure. The estimate uses Hessians of the smooth coordinate
functions through metric compatibility; it does not differentiate
the piecewise connection across $\Sigma$.

\end{proof}

\subsection{The image of the round boundary}

\begin{lemma}\label{lem:marked-boundary}
Let $X:=\{y:|y|\ge1\}$ denote the known Euclidean exterior.
For any fixed sufficiently small dimensional $\eta>0$, choose
$\beta$ sufficiently small relative to $\eta^2$. If
$m_\gamma+\varepsilon$ is sufficiently small, there are a compact
set $A_X\subset X$ and a global bi-Lipschitz map
$F_0:\R^n\to\R^n$ such that
\[
 \Cap(A_X)\le K(m_\gamma+\varepsilon),\qquad
 X\setminus A_X\subset E,\qquad
 F_0=\bar{\Phi}\text{ on }X\setminus A_X,
\]
\[
 \operatorname{Lip}(F_0-\operatorname{Id})\le K\eta<1/100,
 \qquad\operatorname{Lip}(F_0),\operatorname{Lip}(F_0^{-1})\le2.
\]
After a common target translation, the spherical mean of
$F_0-\operatorname{Id}$ is zero, and
\begin{equation}\label{eq:marked-displacement}
 \|F_0-\operatorname{Id}\|_{W^{1,2}(\Sph^{n-1})}^2\le K(m_\gamma+\varepsilon).
\end{equation}
The image of the unit ball is a radial domain satisfying
\begin{equation}\label{eq:marked-radial-domain}
 \begin{gathered}
 F_0(B)=D_f=\{r\theta:0\le r<1+f(\theta)\},\\
 \|f\|_{W^{1,\infty}}\le K\eta,\qquad
 \|f\|_{H^1}^2\le K(m_\gamma+\varepsilon).
 \end{gathered}
\end{equation}
\end{lemma}
\begin{proof}
We anchor the coordinates to the known exterior by controlling the
variation of their differential there. A potential enlargement of the
bad set will turn local derivative control into a pairwise Lipschitz
estimate, allowing extension of the displacement across the ball. The matrix $D_y\bar{\Phi}$ records the covectors
$d\bar{\Phi}_i$ in the known Cartesian coordinates $y$. Multiply this matrix
by a cutoff, setting
\[
 \mathscr C:=\chi_0(S)D_y\bar{\Phi},
\]
where $\chi_0=1$ below $4\beta$ and $\chi_0=0$ above $5\beta$.
This is a bounded matrix field on all of $X$: where its cutoff is
nonzero, the joint defect controls $D_y\bar\Phi$. No exterior
coordinate matrix is being asserted on $M$. It equals
$D_y\bar{\Phi}$ on $E\cap X$, tends to $I$ at infinity,
and \eqref{eq:marked-joint-energy} gives
\[
 \int_X|D_y\mathscr C|^2\dd y\le K(m_\gamma+\varepsilon).
\]
Indeed, where the cutoff is nonzero, $d\bar{\Phi}$ is bounded and the product
rule bounds its derivative by $|D_y^2\bar{\Phi}|+K|dS|$. Since $h=\delta$
on $X$, the covariant Hessian $\Hess_h\bar{\Phi}$ in
\eqref{eq:marked-joint-energy} is exactly $D_y^2\bar{\Phi}$ there, and
$dV_h=\dd y$. This use of the known exterior is essential to the
boundary identification.

Combine the cutoff on $N$ in
\eqref{eq:marked-two-capacities} with a fixed threshold cutoff of
$|\mathscr C-I|$. This produces a compactly supported
Lipschitz function $b_0:X\to[0,1]$, of energy at most $K(m_\gamma+\varepsilon)$,
equal to one near a compact set $K_X$ containing all points of $X\setminus E$
and every point where $|D_y\bar{\Phi}-I|>\eta$.
Use thresholds strictly below $\eta$ in the cutoff, so its plateau
persists under approximation. Thus
\begin{equation}\label{eq:marked-source-good}
 X\setminus K_X\subset E,\qquad
 |D_y\bar{\Phi}-I|\le\eta\quad\text{on }X\setminus K_X.
\end{equation}

The estimate for $b_0$ holds on the exterior $X$. To apply
whole space potential estimates, extend this scalar function into a
reference Euclidean ball; this reference ball is not being identified
with the unknown fill-in. On
$1/2\le r\le1$ put
\[
 a_X(r,\theta):=\min\{1,4r-2\}\,b_0(2-r,\theta),
\]
keep $b_0$ for $r\ge1$, and use zero for $r\le1/2$.
Reflection bounds its derivative terms. The added zeroth-order term
is controlled by radial Cauchy--Schwarz, since $b_0$ vanishes at
infinity. Hence $a_X$ has energy at most $K(m_\gamma+\varepsilon)$ and equals one near
$K_X$. It may be smoothed while preserving that plateau.

A small exceptional set need not be convex, so a line segment
between two good points might still meet it. We enlarge the omitted
set slightly using a potential of the cutoff gradient. This permits
the selection of short paths along which the cutoff stays below one.
For a nonnegative function $F$ write
$I_1F(y):=\int_{\R^n}|y-z|^{1-n}F(z)\dd z$.
For smooth compactly supported $F$, the Fourier multiplier of
$I_1$ is a dimensional multiple of $|\xi|^{-1}$, so Plancherel
and homogeneous Sobolev give
$\|DI_1F\|_2\le K\|F\|_2$; the same estimate holds for
$F\in L^2$ by approximation. If $F\ge0$ and its potential is continuous near a compact set
$\mathcal K$, then, for every $t>0$, the truncation
$\min\{1,2I_1F/t\}$ equals one near
$\mathcal K\cap\{I_1F\ge t\}$. Its gradient energy is bounded by
$Kt^{-2}\|F\|_2^2$. Spatial cutoff and smoothing preserve
this plateau with arbitrarily small excess energy: the cutoff error at
infinity tends to zero by the $L^{2n/(n-2)}$ tail estimate.
Consequently
\begin{equation}\label{eq:marked-potential-capacity}
 \Cap(\mathcal K\cap\{I_1F\ge t\})
       \le Kt^{-2}\|F\|_2^2.
\end{equation}
See also \cite{AdamsHedberg} for the homogeneous Sobolev formulation.
Also the representation
formula gives $|a_X|\le K I_1|Da_X|$.
Put
\[
 P_X:=I_1|Da_X|,\qquad A_X:=X\cap\{P_X\ge\delta_X\},
\]
where $\delta_X>0$ is a small fixed number. Then $A_X$ is compact
and has capacity at most $K(m_\gamma+\varepsilon)$. Since $a_X=1$
near $K_X$, the representation estimate gives $P_X\ge1/K$ there.
Taking $\delta_X<1/K$ yields
\[
 K_X\subset A_X\subset X,\qquad
 X\setminus A_X\subset X\setminus K_X\subset E.
\]
The scalar cutoff has been extended into the reference ball, but
neither exceptional set has been extended into the fill-in.
The inclusion uses the nonlocal potential of $|Da_X|$, which need
not vanish where $a_X$ is constant.

We claim that the displacement $\bar{\Phi}-\operatorname{Id}$ has Lipschitz
constant $K(n)\eta$ on $X\setminus A_X$. This is a statement about
pairs of good points, not only about the derivative at those points.
To prove it locally,
let $x,y$ belong to that set and put $r:=|x-y|\le1/10$.
Let $Z$ be the ball of radius $r/4$ centered at
\[
 \frac{x+y}{2}+4r\,\frac{x+y}{|x+y|}.
\] Each $z\in Z$ is an intermediate point for which both
segments $[x,z]$, $[z,y]$ remain outside the unit ball, and their
total length is at most $10r$. Polar integration gives, for a
nonnegative $F$,
\[
 \frac1{|Z|}\int_Z\int_{[x,z]}F\dd s\dd z
       \le K\int_{B(x,Kr)}|x-y'|^{1-n}F(y')\dd y'.
\]
Apply this at both endpoints with $F=|Da_X|$. Some broken segment
has total cutoff variation at most $K\delta_X$. Its endpoint
values are also at most $K\delta_X$. For sufficiently small
$\delta_X$, the cutoff stays below one along the whole path, which
therefore avoids $K_X$. It need not avoid $A_X$: the derivative
estimate holds on $X\setminus K_X$, whereas $X\setminus A_X$
supplies endpoints admitting such paths. Integrating
\eqref{eq:marked-source-good}
along it proves the claim for nearby pairs. For distant pairs, put $\ell:=|x-y|>1/10$ and fix
$s:=1/1000$. Every exterior ball satisfies, for a fixed $c_{\mathrm{vol}}>0$ depending only on $n$,
$|B_s(y_*)\cap X|\ge c_{\mathrm{vol}}s^n$: the half-ball beyond a supporting
hyperplane of the unit sphere lies in $X$. Choose the mass and
approximation error so small that $|A_X|<c_{\mathrm{vol}}s^n$.
There is a curve in $X$ from $x$ to $y$ of length at most $2\ell$.
Indeed, replace any portion of their straight segment inside the ball
by the shorter spherical arc; its arc-to-chord ratio is at most
$\pi/2$. Partition the curve into at most $2\ell/s+1$ pieces of
length at most $s$. Each interior partition point has a point of
$X\setminus A_X$ within distance $s$, by the volume bound.
Keep $x$ and $y$ as the endpoints. Consecutive selected points are
at distance at most $3s<1/10$, and the sum of those distances is
at most $K(n)\ell$. Summing the local estimate proves the claimed
ambient Lipschitz bound for every pair of good points.

Apply Kirszbraun's extension theorem \cite{Kirszbraun} to the displacement,
obtaining
a map $F_0:\R^n\to\R^n$ such that $F_0-\operatorname{Id}$ extends
$\bar{\Phi}-\operatorname{Id}$ from $X\setminus A_X$ with the same
Lipschitz constant. Choose $\eta$ so that
$\operatorname{Lip}(F_0-\operatorname{Id})<1/100$.
The contraction equation
$y=\xi-(F_0(y)-y)$ shows that $F_0$ is a global bi-Lipschitz
homeomorphism, with both Lipschitz constants at most two. Its straight homotopy to the
identity preserves orientation, and
\[
 F_0=\bar{\Phi}\quad\text{on }X\setminus A_X.
\]
The differential estimates do not control an additive translation.
We fix this harmless freedom by translating $\bar{\Phi},F_0$ and all target
objects together, and translating $\Phi$ by $r_{\mathrm S}$
times the same vector so that $\bar{\Phi}=\Phi/r_{\mathrm S}$
continues to hold. Choose the translation so that
$\int_{\Sph^{n-1}}(F_0-\operatorname{Id})\dd\sigma=0$. Then
$\sup_{\Sph^{n-1}}|F_0-\operatorname{Id}|
\le2\operatorname{Lip}(F_0-\operatorname{Id})$. The same normalization
gives $|F_0(0)|\le\operatorname{Lip}(F_0-\operatorname{Id})$.
Radial Cauchy--Schwarz applied to $\mathscr C-I$, with
its zero limit at infinity, gives
\[
 \|\mathscr C-I\|_{L^2(\Sph^{n-1})}^2\le K(m_\gamma+\varepsilon).
\]
On $\Sigma\setminus A_X$ this controls the tangential derivative of
$F_0-\operatorname{Id}$; on $\Sigma\cap A_X$ that derivative is bounded by its fixed
Lipschitz constant, and this set has area at most
$K(m_\gamma+\varepsilon)^{(n-1)/(n-2)}$. Spherical Poincar\'e therefore gives
\eqref{eq:marked-displacement}.
For $\theta\in\Sph^{n-1}$ set
\[
 R_{F_0}(\theta):=|F_0(\theta)|,\qquad
 \Pi_{F_0}(\theta):=\frac{F_0(\theta)}{|F_0(\theta)|}.
\]
Then $|R_{F_0}-1|\le\|F_0-\operatorname{Id}\|_\infty$ and
$|\Pi_{F_0}-\theta|\le K\|F_0-\operatorname{Id}\|_\infty$ on the sphere.
The small Lipschitz displacement also gives
$\operatorname{Lip}(\Pi_{F_0}-\operatorname{Id})\le
K\operatorname{Lip}(F_0-\operatorname{Id})$.
Thus $\Pi_{F_0}$ is injective for the chosen tolerance and onto by degree one.
Moreover, the contraction equation and the bound for $|F_0(0)|$ give
\[
 |F_0^{-1}(0)|\le
 \frac{|F_0(0)|}{1-\operatorname{Lip}(F_0-\operatorname{Id})}<1.
\]
Thus $0\in F_0(B)$, where $B$ is the unit ball in the reference
Euclidean source space. Its radial boundary therefore bounds precisely
$F_0(B)$, giving \eqref{eq:marked-radial-domain}.
Specifically, $f(\Pi_{F_0}(\theta))=R_{F_0}(\theta)-1$. The uniform Lipschitz
bound follows from the small displacement, and the $H^1$ estimate
follows by differentiating this identity and using
\eqref{eq:marked-displacement}. Thus $F_0$
parametrizes $\partial D_f$ by the original round sphere.

\end{proof}

\subsection{Identifying the interior}

\begin{lemma}\label{lem:marked-interior}
For the sets and maps in Lemmas~\ref{lem:marked-sheet}
and~\ref{lem:marked-boundary}, there is a compact set
$L\supset K_0\cup F_0(A_X)$ with
$\Cap(L)\le K(m_\gamma+\varepsilon)$ such that
\begin{equation}\label{eq:interior-image}
 \bar{\Phi}\bigl(E\cap\operatorname{int}M
       \cap \bar{\Phi}^{-1}(\R^n\setminus L)\bigr)=D_f\setminus L.
\end{equation}
On $\partial D_f\setminus L$, the inverse of
$\bar{\Phi}$ agrees with $F_0^{-1}$ and takes values in
$\Sigma$. Thus $D_f\setminus L$ corresponds to points of $\operatorname{int}M$, while
$\partial D_f\setminus L$ corresponds to a subset of $\Sigma$.
\end{lemma}
\begin{proof}
Although $F_0$ is a homeomorphism of Euclidean space, it was
constructed by extending the displacement and is not a coordinate
map on $M$. We must therefore show that
$\bar{\Phi}^{-1}$ maps most of $D_f$ into $M$. Merely knowing the
location of its outer boundary does not establish this assertion.

Put $W_0:=K_0\cup F_0(A_X)$. Its capacity is at most
$K(m_\gamma+\varepsilon)$, by the bi-Lipschitz bound for $F_0$. Outside
this set, the inverse exists and the map $F_0$ agrees with $\bar{\Phi}$ at every
corresponding point of $X\setminus A_X$. Choose a smooth
compact cutoff $a_W=1$ near $W_0$ of energy at most
$K(m_\gamma+\varepsilon)$. On $\bar{\Phi}(E\cap X)$, the image of the part of
the exterior in $E$, define
\[
 \mathcal F(\xi):=\bigl|D_\xi(D_y\bar{\Phi}\circ \bar{\Phi}^{-1})(\xi)\bigr|,
\]
and extend this scalar function by zero elsewhere. Here the inner
derivative $D_y\bar{\Phi}$ is computed in the known exterior coordinate,
whereas $D_\xi$ differentiates its matrix entries in target
coordinates. The chain rule expresses this derivative as
$D_y^2\bar{\Phi}$ multiplied by the derivative of the inverse.
Thus the exterior Hessian bound and bounded distortion give
$\|\mathcal F\|_2^2\le K(m_\gamma+\varepsilon)$.
Let
\[
 P_T:=I_1(\mathcal F+|Da_W|).
\]
For every $t>0$, \eqref{eq:marked-potential-capacity} bounds the capacity
of each compact subset of $\{P_T\ge t\}$ by $Kt^{-2}(m_\gamma+\varepsilon)$. The function is continuous on
compact sets: $\mathcal F$ is locally bounded, by properness and
the smooth exterior map with uniformly bounded inverse derivatives;
the local Riesz kernel is integrable against bounded data, and its
distant part is continuous in $L^2$ for $n>2$.

We prove the required exclusion by contradiction, using a target
ray whose lift moves strictly outwards in the exterior $X$.
Suppose a point $y_0\in E\cap X$
has image
$\xi_0\in D_f\setminus W_0$ with $P_T(\xi_0)<\delta_T$.
Write $\mathscr C_0:=D_y\bar{\Phi}(y_0)$ and
$\vartheta_0:=y_0/|y_0|$. The Gram estimate ensures that the singular
values of $\mathscr C_0$ lie in $[1/2,2]$; it does not assert
$\mathscr C_0\approx I$ at this point. Put
$e_*:=\mathscr C_0\vartheta_0/|\mathscr C_0\vartheta_0|$.
For every unit target direction with $|e-e_*|\le1/8$,
$\vartheta_0\cdot\mathscr C_0^{-1}e_*=1/|\mathscr C_0\vartheta_0|\ge1/2$
and $\|\mathscr C_0^{-1}\|\le2$, so
\[
 |\mathscr C_0^{-1}e|\ge\tfrac12,\qquad
 \frac{y_0}{|y_0|}\cdot\mathscr C_0^{-1}e\ge\tfrac14.
\]
This spherical cap has a fixed positive dimensional measure.
Its elements are target vectors: the inverse differential converts
them to source velocities. Polar integration selects one such $e$ with
\[
 \int_0^\infty(\mathcal F+|Da_W|)(\xi_0+se)\dd s
       \le K\delta_T.
\]
Since $a_W(\xi_0)\le KP_T(\xi_0)$, for small $\delta_T$ the cutoff
stays below one along this ray. The ray avoids $W_0$, so its
lift exists for all finite times. As long as it is in the exterior $X$, write its known
exterior coordinate as
$y(s):=\bar{\Phi}^{-1}(\xi_0+se)$. On an exterior interval,
\[
 \left|D_y\bar\Phi(y(s))-\mathscr C_0\right|
 \le\int_0^s\mathcal F(\xi_0+te)\,\dd t\le K\delta_T.
\]
This is differentiation of the exterior coframe in target coordinates;
the speed of the target ray is one.
Differentiating $\bar{\Phi}(y(s))=\xi_0+se$ therefore shows that the velocity
is close to its initial value. With $\mathbf a:=\mathscr C_0^{-1}e$ this gives
\[
 |\dot y-\mathbf a|\le\tfrac18.
\]
Put $b(s):=y(s)\cdot\mathbf a/|y(s)|$. Direct differentiation gives
\[
 |y|\dot b=|\mathbf a|^2-b^2+\mathbf a\cdot(\dot y-\mathbf a)
                    -b\frac{y}{|y|}\cdot(\dot y-\mathbf a).
\]
At a hypothetical first downward crossing of $b=1/4$,
\[
 |y|\dot b\ge |\mathbf a|^2-\frac{|\mathbf a|}{8}-\frac{3}{32}
                    \ge\frac{3}{32}>0,
\]
because $|\mathbf a|\ge1/2$. Such a crossing is impossible.
Hence $d|y|/ds\ge1/8$ throughout its exterior
interval. A lift starting on $\Sigma$ enters the exterior by this
positive one-sided derivative. The selected inverse is smooth across
$\Sigma$ as a map of the underlying smooth manifold, so this derivative
indeed determines its initial side. A first exit argument shows that it
remains exterior forever. There is no finite breakdown of the lift,
since each finite ray segment avoids $K_0\subset W_0$ and is compact
in $Y$. Its image under the continuous selected inverse is compact
in $E$; it cannot escape to a different branch.

The ray must leave the bounded domain $D_f$. At its first boundary
crossing it avoids $F_0(A_X)$, so the crossing equals
$F_0(\theta)=\bar{\Phi}(\theta)$ for a point $\theta\in\Sigma\setminus A_X$.
Uniqueness of the inverse then makes the lift equal to
$\theta$, contradicting its strictly increasing radius. We have proved
\begin{equation}\label{eq:marked-ray-exclusion}
 \bar{\Phi}(E\cap X)\cap D_f
       \subset W_0\cup\{P_T\ge\delta_T\}.
\end{equation}

Choose a compact neighborhood of
\[
 W_0\cup\bigl(\overline{D_f}\cap\{P_T\ge\delta_T/2\}\bigr)
\]
with capacity at most $K(m_\gamma+\varepsilon)$, fill its bounded
complementary components, and call it $L$. The lower threshold
$\delta_T/2$ leaves room to take a neighborhood of the set where
the exclusion might fail. More precisely, continuity of $P_T$ gives
\[
 \overline{\bar\Phi(E\cap X)\cap D_f}
 \subset W_0\cup\bigl(\overline{D_f}\cap\{P_T\ge\delta_T\}\bigr).
\]
The lower threshold provides an admissible plateau around the compact
right-hand side, to which \eqref{eq:marked-potential-capacity} applies.
This is why taking a closure is legitimate here; no general claim that
closure preserves capacity is used.

We can now identify which points correspond to the interior and exterior of $M$. If
$\xi\notin\overline{D_f}\cup W_0$, the point $F_0^{-1}(\xi)$
lies in the exterior $X$ and outside $A_X$. Hence
$F_0=\bar{\Phi}$ there, and uniqueness identifies this point with the
inverse of $\xi$. If $\xi\in\partial D_f\setminus W_0$,
the same argument identifies its inverse with the corresponding point of
$\Sigma\setminus A_X$. Finally, if
$\xi\in D_f\setminus L$, the inverse exists because
$K_0\subset L$, and it cannot be exterior or on $\Sigma$ by
\eqref{eq:marked-ray-exclusion}. It therefore lies in
$\operatorname{int}M$. Together these three observations establish
\eqref{eq:interior-image}.
No topological assumption is made on the part of $M$ omitted by this construction.

\end{proof}

\subsection{Completion of the coordinate construction}

\begin{proof}[Proof of Proposition~\ref{prop:marked}]

Apply Lemmas~\ref{lem:marked-joint}--\ref{lem:marked-interior}
with the tolerances chosen below. We construct a tensor on all of $D_f$ that agrees with
the metric perturbation on $O$.
Choose a smooth capacity cutoff $a_L=1$ near $L$, with values in
$[0,1]$ and energy at most $K(m_\gamma+\varepsilon)$. Let $\vartheta$ equal one on
$[0,1/4]$ and zero on $[1/2,1]$. Set
\[
 q_{\mathrm{ref}}:=\vartheta(a_L)q\quad\text{on }D_f\setminus L,
 \qquad
 A:=\{a_L\ge1/4\}\cap\overline{D_f},\qquad O:=D_f\setminus A,
\]
and extend $q_{\mathrm{ref}}$ by zero near $L$. Since $a_L<1/4$ on $O$,
$q_{\mathrm{ref}}=q$ there. Near $L$ the multiplier vanishes, so extension
by zero creates no jump. In the product rule, the term involving
$Da_L$ is bounded by the small uniform norm of $q$ times
the cutoff energy. Together with
\eqref{eq:marked-metric-estimate}, this gives
\[
 \|q_{\mathrm{ref}}\|_{H^1(D_f)}^2+\|q\|_{H^1(O)}^2\le K(m_\gamma+\varepsilon),
 \qquad \|q_{\mathrm{ref}}\|_\infty\le\|q\|_\infty,
 \qquad q_{\mathrm{ref}}=q\text{ on }O.
\]
The cutoff and \eqref{eq:marked-capacity-size} give
\[
 \Cap(A)\le K(m_\gamma+\varepsilon),\qquad |A|\le K(m_\gamma+\varepsilon)^{n/(n-2)}.
\]
By \eqref{eq:interior-image}, $O$ corresponds to a subdomain of $M$, and $\delta+q$ is the
metric $h$ expressed in these coordinates. In particular, it is scalar flat.
The distinction between $q$ and $q_{\mathrm{ref}}$ will be maintained below:
curvature equations hold for $\delta+q$ on $O$, whereas $q_{\mathrm{ref}}$ supplies
an $H^1$ tensor on all of $D_f$. No scalar curvature assertion is
made about $\delta+q_{\mathrm{ref}}$ in the omitted set.

Let $\Gamma:=\partial D_f\setminus A$. Its preimage in $\Sigma$ is
$F_0^{-1}(\Gamma)\subset\Sigma\setminus A_X$.
Since $F_0(\Sigma)=\partial D_f$, the omitted source set is exactly
\[
 \Sigma\setminus F_0^{-1}(\Gamma)=\Sigma\cap F_0^{-1}(A).
\]
Capacity distortion and the fixed sphere trace estimate bound its area
by $K(m_\gamma+\varepsilon)^{(n-1)/(n-2)}$. The omitted target boundary
is its image under $F_0$, so the Lipschitz area inequality gives the
same bound for its Euclidean area. The area $\operatorname{Area}_h(\Gamma)$ is exactly
$|\Sph^{n-1}|$ minus the omitted source area, because $h_\Sigma=\sigma$.
No artificial inner boundary enters this identity.

Finally, we transfer the estimates for the harmonic function $v^{-1}$
to $O$. Since $v$ is uniformly close to one on $E$, $|v^{-1}-1|\le2$ there after
fixing the small tolerances. Bounded distortion and
\eqref{eq:inverse-v-energy} bound the squared $L^2$ norm of the gradient
on $O$ by $K(m_\gamma+\varepsilon)$. On $\Gamma$, the boundary
identification $F_0^*h=\sigma$ and the boundary estimate in
\eqref{eq:inverse-v-energy} give the same bound for the squared
$L^2$ norm of $v^{-1}-1$. What remains is an
interior $L^2$ bound; it cannot be obtained by assuming a uniform
Poincar\'e inequality on the original fill-in.
Choose a capacity cutoff $0\le\chi\le1$, equal to one near $A$, with
$\|\chi\|_{H^1(B_2)}\le K\sqrt{m_\gamma+\varepsilon}$.
The zero extension of $(1-\chi)(v^{-1}-1)$ to $D_f$ has gradient norm
$O(\sqrt{m_\gamma+\varepsilon})$ and boundary trace norm $O(\sqrt{m_\gamma+\varepsilon})$.
The uniform trace--Poincar\'e inequality
\[
 \|b\|_{L^2(D_f)}\le K\bigl(\|Db\|_{L^2(D_f)}
       +\|\operatorname{Tr}b\|_{L^2(\partial D_f)}\bigr)
\]
follows by pulling back through $F_0$ from the unit ball. It bounds
the $L^2$ norm of this zero extension by $K\sqrt{m_\gamma+\varepsilon}$.
The remaining product satisfies
$\|\chi(v^{-1}-1)\|_2\le2\|\chi\|_2\le K\sqrt{m_\gamma+\varepsilon}$.
This proves the $H^1(O)$ bound for $v^{-1}-1$ without a Sobolev inequality on the
omitted part of $M$.

Choose the tolerances $\eta$ and $\beta$ above in terms of $\tau$,
then choose the smallness bound for $C$. These are dimensional
choices, since $m_\gamma\le C/[2n(n-1)]$ and $\varepsilon$ is arbitrary.
The scalar curvature equations, harmonicity equations, and boundary identities hold for
the original metrics and functions. Smooth approximation was used only to obtain the
graph map and smooth regular level sets of $S$.

\end{proof}

\section{The weighted integral estimate}\label{sec:weighted}
Proposition~\ref{prop:marked} supplies coordinates on most of the fill-in,
a domain close to the Euclidean unit ball, and a deleted set of small
capacity. A direct continuity estimate would lose a square root of the
energy. The mass argument requires instead an error of order
$\mathscr E$. We obtain it by identifying the first variation of the
conformal integral and using scalar flatness and the prescribed boundary
area to cancel it. A separate harmonic calculation controls the weight
$v^{-1}$.

All estimates below are uniform for sufficiently small $C>0$. In
particular, the model and its adjoint are kept exact in $C$: no error
depending only on a positive power of $C$ is introduced. This is needed
when the auxiliary approximation parameter tends to zero at fixed $C$.
Capacity cutoffs remove the deleted set without imposing any regularity
on its boundary. When comparing equations, we subtract them before
estimating cutoff derivatives, so that those derivatives pair with
another small quantity.

\subsection{The round model and a small comparison solution}
Unless a metric is indicated, norms, derivatives, contractions, and
measures in this section are Euclidean. The symbol $B$ denotes the
Euclidean unit ball, and $z$ is the parameter in
\eqref{eq:parameters}.

\begin{lemma}\label{lem:round-model}
For
\[
 \alpha:=\frac{1-z}{z},\qquad
 v_C(r):=\left(\frac{1+\alpha}{1+\alpha r^2}\right)^{(n-2)/2},
 \qquad 0\le r\le2,
\]
one has
\begin{gather*}
 -\frac{4(n-1)}{n-2}\Delta v_C
       =Cv_C^{\frac{n+2}{n-2}},\qquad
 v_C|_{\partial B}=1,\qquad
 v_C'(1)=-\frac{(n-2)\alpha}{1+\alpha},\\
 \int_Bv_C^{\frac{n+2}{n-2}}\dd x
       =\frac{|\Sph^{n-1}|}{nz}.
\end{gather*}
The metric $g_C:=v_C^{\frac4{n-2}}\delta$ on $\overline B$ is the
small round cap of scalar curvature $C$ and unit round boundary.
For every fixed nonnegative integer $j$,
$\|v_C-1\|_{C^j(B_2)}\le K_jC$.
\end{lemma}

\begin{proof}
The identities
\[
 C=\frac{4n(n-1)\alpha}{(1+\alpha)^2},\qquad
 1+\alpha=\frac1z
\]
give the equation and boundary derivative by differentiation; integration
of the equation gives the integral formula. In the coordinates
$y=\sqrt\alpha\,x$,
\[
 g_C=\frac{(1+\alpha)^2}{\alpha(1+|y|^2)^2}|\dd y|^2.
\]
This is a stereographic round metric of sectional curvature
$C/[n(n-1)]$. Its angular radius is $2\arctan\sqrt\alpha<\pi/2$,
and its boundary is unit round because $v_C(1)=1$. The derivative
bounds follow from $\alpha=O(C)$ and the explicit formula.
\end{proof}

For an arbitrary open set $\Omega$, $H^1_0(\Omega)$ denotes the
closure of $C_c^\infty(\Omega)$ in $H^1$. Its elements can be extended
by zero to any larger open set. The following estimates therefore apply
to domains with irregular holes.

\begin{lemma}\label{lem:torsion}
Let $\Omega\subset B_2$ be open, and let $h$ be a measurable metric
with $\|h-\delta\|_{L^\infty(\Omega)}$ sufficiently small, depending
only on $n$. The weak solution of
\[
 -\Delta_h\Theta=1,\qquad \Theta\in H^1_0(\Omega),
\]
is unique and satisfies
\begin{equation}\label{eq:torsion-bound}
 0\le\Theta,\qquad
 \|\Theta\|_\infty
 \le\frac2n\left(\frac{|\Omega|}{|B|}\right)^{2/n}
 \le\frac8n.
\end{equation}
Here $|\Omega|$ is Euclidean volume.
\end{lemma}

\begin{proof}
Choose the amplitude so that both the volume density and the Dirichlet
integrand lie between one half and twice their Euclidean counterparts.
Coercivity, using zero-extension Poincar\'e, gives existence and
uniqueness. Testing with the negative part gives nonnegativity. If
$\mathcal V(t):=|\{\Theta>t\}|$, truncation, coarea,
Cauchy--Schwarz, and the Euclidean isoperimetric inequality give
\[
 \int_{\{\Theta=t\}}|D\Theta|\dd A\le4\mathcal V(t),
 \qquad
 -\mathcal V'(t)\ge
 \frac{n^2|B|^{2/n}}4\mathcal V(t)^{(n-2)/n}
\]
for almost every positive level. Integration proves
\eqref{eq:torsion-bound}. Approximation of the coefficients and
exhaustion of the domain justify this calculation for the stated weak
solution; the constants use only the displayed ellipticity bounds.
\end{proof}

\begin{lemma}\label{lem:comparison-solution}
Under the hypotheses of Lemma~\ref{lem:torsion}, for sufficiently small
$C>0$ there is a solution
\begin{equation}\label{eq:branch}
 -\frac{4(n-1)}{n-2}\Delta_hw_0
       =Cw_0^{\frac{n+2}{n-2}},\qquad
 w_0-1\in H^1_0(\Omega),\qquad 1\le w_0\le1+KC.
\end{equation}
It is unique among solutions in this interval. All constants depend
only on $n$. If $V\in H^1(\Omega)$, $V\ge1$, and
\[
 -\frac{4(n-1)}{n-2}\Delta_hV
       \ge CV^{\frac{n+2}{n-2}}
\]
in distributions, with the power locally integrable, then $w_0\le V$.
\end{lemma}

\begin{proof}
The functions $1$ and
$1+\frac{(n-2)C}{2(n-1)}\Theta$ are respectively a subsolution
and a supersolution for small $C$. Monotone iteration gives
\eqref{eq:branch}. The derivative of
$s\mapsto s^{\frac{n+2}{n-2}}$ is uniformly bounded on
$[1,1+KC]$; subtraction and testing with the difference give uniqueness
by Poincar\'e and absorption of a term bounded by $KC$ times the
squared $L^2$ norm. The same argument gives uniform coercivity for
the linearized equations used below.

For the comparison with a possibly large supersolution, start the
iteration at $w^{(0)}=1$. If $w^{(j)}\le V$, the distributional
inequality implies
\[
 -\frac{4(n-1)}{n-2}\Delta_hV
       \ge C\bigl(w^{(j)}\bigr)^{\frac{n+2}{n-2}}.
\]
This inequality extends to nonnegative $H^1_0(\Omega)$ tests: after
replacing its right side by the bounded iterate, both sides are
continuous functionals on that space. The function
$t=(w^{(j+1)}-V)_+$ belongs to $H^1_0(\Omega)$ because
$0\le t\le w^{(j+1)}-1\in H^1_0(\Omega)$ and this space is an
order ideal in $H^1(\Omega)$. Explicitly, compactly supported
nonnegative approximations $t_i$ to $w^{(j+1)}-1$ give
$\min\{t,t_i\}\to t$ in $H^1$. Testing the subtracted equations
with $t$ yields
\[
 \frac{4(n-1)}{n-2}\int_\Omega|\nabla^h t|_h^2\dd V_h\le0.
\]
Induction and passage to the monotone limit prove $w_0\le V$.
\end{proof}

We apply this lemma on $O$ with the metric $h$ in the selected
coordinates. Its small metric amplitude is supplied by
Proposition~\ref{prop:marked}. The conformal factor $V$ lies
in $H^1(O)$, since it is smooth on the fixed compact fill-in and the
coordinate distortion on $O$ is bounded. Thus the resulting $w_0$
satisfies $w_0\le V$.

\subsection{The fixed domain expansion and its exact adjoint}
The next two lemmas isolate the analytic calculation on the fixed
reference ball. They will be applied after the moving domain has been
pulled back to $B$.

\begin{lemma}\label{lem:solution-taylor}
Let $A_0\subset\overline B$ be compact, put
$\Omega=B\setminus A_0$, and let $h_1=\delta+q_1$ with
$\|q_1\|_\infty$ sufficiently small. Let $w_1$ be the solution
from Lemma~\ref{lem:comparison-solution} for this domain and metric,
and set
\[
 e:=(w_1-1)_0-(v_C-1)\in H^1_0(B),
\]
where the subscript $0$ means zero extension from $\Omega$. Then
\begin{equation}\label{eq:standalone_state}
 \|e\|_{H^1(B)}\le
 KC\left(\|q_1\|_{L^2(\Omega)}+\Cap(A_0)^{1/2}\right),
\end{equation}
and
\begin{equation}\label{eq:integral-taylor}
\begin{aligned}
 \int_\Omega w_1^{\frac{n+2}{n-2}}\dd V_{h_1}
       -\frac{|\Sph^{n-1}|}{nz}
 ={}&\frac12\int_\Omega(\tr q_1)
                   v_C^{\frac{n+2}{n-2}}\dd x\\
 &+\frac{n+2}{n-2}\int_Bv_C^{\frac4{n-2}}e\dd x\\
 &+O\!\left(\|q_1\|_{L^2(\Omega)}^2+\Cap(A_0)\right).
\end{aligned}
\end{equation}
\end{lemma}

\begin{proof}
Choose a smooth cutoff $0\le\chi_0\le1$, equal to one near $A_0$,
with
\[
 \|\chi_0\|_{H^1(B_2)}^2\le K\Cap(A_0).
\]
An arbitrarily small excess is understood and removed after the
estimates. The $L^2$ part follows from homogeneous Sobolev and H\"older
on $B_2$. The correction obeys
\[
 \|\chi_0(v_C-1)\|_{H^1(B)}
       \le KC\Cap(A_0)^{1/2},
\]
and $(1-\chi_0)(v_C-1)\in H^1_0(\Omega)$.

For an admissible test $\varphi\in H^1_0(\Omega)$, its zero extension
belongs to $H^1_0(B)$ and can therefore test the model equation.
Subtracting the model equation from the equation for $w_1$ gives
\begin{align*}
 &\frac{4(n-1)}{n-2}\int_\Omega
   \sqrt{\det h_1}\,(h_1^{-1}De)\cdot D\varphi\dd x\\
 &\quad=C\int_\Omega\sqrt{\det h_1}
   \left(w_1^{\frac{n+2}{n-2}}-v_C^{\frac{n+2}{n-2}}\right)
   \varphi\dd x\\
 &\qquad+C\int_\Omega(\sqrt{\det h_1}-1)
                        v_C^{\frac{n+2}{n-2}}\varphi\dd x\\
 &\qquad-\frac{4(n-1)}{n-2}\int_\Omega
   (\sqrt{\det h_1}\,h_1^{-1}-I)Dv_C\cdot D\varphi\dd x.
\end{align*}
The last two terms are bounded by
$KC\|q_1\|_2\|\varphi\|_{H^1}$, because $Dv_C=O(C)$.
Use the particular test
\[
 \varphi=(w_1-1)-(1-\chi_0)(v_C-1)\in H^1_0(\Omega).
\]
On all of $B$, $e=\varphi_0-\chi_0(v_C-1)$. The derivative of
the nonlinear power is uniformly bounded on the common small interval.
Ellipticity and Poincar\'e absorb the resulting
$KC\|\varphi\|_2^2$ term, while the cutoff correction costs at most
$KC\Cap(A_0)^{1/2}\|\varphi\|_{H^1}$. This proves
\eqref{eq:standalone_state}.

On $\Omega$, Taylor expansion gives
\[
 \sqrt{\det h_1}\,w_1^{\frac{n+2}{n-2}}
 =v_C^{\frac{n+2}{n-2}}
  +\frac12(\tr q_1)v_C^{\frac{n+2}{n-2}}
  +\frac{n+2}{n-2}v_C^{\frac4{n-2}}e
  +O(|q_1|^2+e^2).
\]
The derivatives of the power are bounded on this positive interval in
every dimension. The integrated remainder is controlled by
\eqref{eq:standalone_state}. The omitted model integral over
$A_0\cap B$ costs $K|A_0\cap B|$, and extending the linear $e$
integral to $B$ costs $KC|A_0\cap B|$, because $e=1-v_C$ there.
The isocapacitary inequality, on this fixed bounded set, gives
$|A_0\cap B|\le K\Cap(A_0)$. This proves
\eqref{eq:integral-taylor}.
\end{proof}

The linear solution error in \eqref{eq:integral-taylor} is generally
only of order $C(\|q_1\|_2+\Cap(A_0)^{1/2})$. It cannot be
discarded as a quadratic error. An adjoint equation expresses it as
another linear metric term.

\begin{lemma}\label{lem:first-variation}
Under the hypotheses of Lemma~\ref{lem:solution-taylor}, define the
following smooth radial functions and tensor on $B_2$:
\begin{align}
 \psi(r)&:=\frac1C\left[
 \frac{(1-\alpha r^2)(1+\alpha)}{(1+\alpha r^2)(1-\alpha)}
 v_C(r)-1\right],
 \label{eq:standalone_adjoint}\\
 \zeta(r)&:=\frac1{2(n-1)}\int_r^1 s
 \left[\frac{4(n-1)}{n-2}\psi'(s)v_C'(s)
       +(1+C\psi(s))v_C(s)^{\frac{n+2}{n-2}}\right]\dd s,
 \label{eq:standalone_zeta}\\
 T&:=D^2\zeta-(\Delta\zeta)I.
 \label{eq:standalone_stresspotential}
\end{align}
Then
\begin{equation}\label{eq:standalone_coefficient}
 \int_\Omega w_1^{\frac{n+2}{n-2}}\dd V_{h_1}
       -\frac{|\Sph^{n-1}|}{nz}
 =\int_\Omega T:q_1\dd x
  +O\!\left(\|q_1\|_{L^2(\Omega)}^2+\Cap(A_0)\right).
\end{equation}
Moreover, $\zeta(1)=0$ and $T_{rr}(1)=-(n-1)\zeta'(1)$.
Every fixed number of derivatives of $\psi$, $\zeta$, and $T$ is
uniformly bounded as $C\downarrow0$. The tensor $T$ is a radial
tensor field; it is not in general constant.
\end{lemma}

\begin{proof}
The explicit adjoint satisfies
\begin{equation}\label{eq:adjoint-equation}
 \left(-\frac{4(n-1)}{n-2}\Delta
       -\frac{(n+2)C}{n-2}v_C^{\frac4{n-2}}\right)\psi
 =\frac{n+2}{n-2}v_C^{\frac4{n-2}},
 \qquad \psi|_{\partial B}=0.
\end{equation}
Indeed, differentiation of the dilation family for the model equation
shows that
$\frac{1-\alpha r^2}{1+\alpha r^2}v_C$ solves the homogeneous
linearized equation. Normalize its boundary value to one, subtract
one, and divide by $C$ to obtain \eqref{eq:standalone_adjoint}.
Coercivity gives uniqueness. Testing the adjoint equation with
$e\in H^1_0(B)$ yields the exact identity
\begin{equation}\label{eq:adjoint-testing}
\begin{aligned}
 \frac{n+2}{n-2}\int_Bv_C^{\frac4{n-2}}e\dd x
 ={}&\frac{4(n-1)}{n-2}\int_BDe\cdot D\psi\dd x\\
 &-\frac{(n+2)C}{n-2}\int_Bv_C^{\frac4{n-2}}e\psi\dd x.
\end{aligned}
\end{equation}

Take the common test
$\varphi=(1-\chi_0)\psi\in H^1_0(\Omega)$ in the equations
for $w_1$ and $v_C$, and subtract. Using
\[
 \sqrt{\det h_1}\,h_1^{-1}
 =I+\tfrac12(\tr q_1)I-q_1+O(|q_1|^2)
\]
and keeping the linear terms gives
\begin{equation}\label{eq:linearized-comparison}
\begin{aligned}
 &\frac{4(n-1)}{n-2}\int_BDe\cdot D\varphi\dd x
 -\frac{(n+2)C}{n-2}\int_Bv_C^{\frac4{n-2}}e\varphi\dd x\\
 &\quad=\frac C2\int_\Omega(\tr q_1)
                        v_C^{\frac{n+2}{n-2}}\varphi\dd x\\
 &\qquad-\frac{4(n-1)}{n-2}\int_\Omega
       \left[\tfrac12(\tr q_1)I-q_1\right]
                         Dv_C\cdot D\varphi\dd x\\
 &\qquad+O\!\left(\|q_1\|_2^2+\Cap(A_0)\right).
\end{aligned}
\end{equation}
The left integrals are over $B$ because $\varphi$ vanishes near
the hole. The positive $C/2$ term is the source volume variation;
the negative Dirichlet term is moved from the left side.

For completeness, the source remainders are bounded by
\[
 KC\int_\Omega(|q_1|^2+e^2+|q_1||e|)\dd x.
\]
The omitted gradient terms are bounded by the integral of
$K(|q_1||De|+|q_1|^2|Dv_C|)|D\varphi|$.
The bounded part $(1-\chi_0)D\psi$ contributes at most
\[
 K\|q_1\|_2\|De\|_2+KC\|q_1\|_2^2,
\]
and the part $-\psi D\chi_0$ contributes at most
\[
 K\|q_1\|_\infty
      (\|De\|_2+C\|q_1\|_2)\|D\chi_0\|_2.
\]
The solution estimate and the capacity bound control these by the
remainder in \eqref{eq:linearized-comparison}.

We may now replace $\varphi$ by $\psi$. The correction on the left
is bounded on the whole ball by
\[
 K\|e\|_{H^1(B)}\|\chi_0\psi\|_{H^1(B)}
 \le KC\left(\|q_1\|_2^2+\Cap(A_0)\right);
\]
the two linear metric corrections obey the same bound. This estimate
already includes the zero-extension error inside the hole, so no
additional hole integral is needed. By \eqref{eq:adjoint-testing},
\begin{equation}\label{eq:adjoint-metric-term}
\begin{aligned}
 \frac{n+2}{n-2}\int_Bv_C^{\frac4{n-2}}e\dd x
 ={}&\frac C2\int_\Omega(\tr q_1)
                         v_C^{\frac{n+2}{n-2}}\psi\dd x\\
 &-\frac{4(n-1)}{n-2}\int_\Omega
       \left[\tfrac12(\tr q_1)I-q_1\right]
                         Dv_C\cdot D\psi\dd x\\
 &+O\!\left(\|q_1\|_2^2+\Cap(A_0)\right).
\end{aligned}
\end{equation}
This is the reason to subtract before estimating: each cutoff
derivative pairs with $e$ or $q_1$, giving capacity rather than its
square root.

Combining \eqref{eq:adjoint-metric-term} and
\eqref{eq:integral-taylor} gives contraction with
\begin{equation}\label{eq:standalone_stress}
 \begin{aligned}
 \widehat T={}&\frac{4(n-1)}{n-2}\operatorname{sym}(D\psi\otimes Dv_C)\\
 &+\left[\frac12v_C^{\frac{n+2}{n-2}}
       -\frac{2(n-1)}{n-2}\operatorname{div}(\psi Dv_C)\right]I.
 \end{aligned}
\end{equation}
Here $\operatorname{sym}$ is the symmetric part, including the
factor $1/2$. Expanding the divergence and using the model equation
recovers the two metric terms above. Direct differentiation gives
\[
 \operatorname{div}\widehat T
 =\frac{2(n-1)}{n-2}
   \bigl[(\Delta\psi)Dv_C-\psi D(\Delta v_C)\bigr]
   +\frac{n+2}{2(n-2)}v_C^{\frac4{n-2}}Dv_C=0,
\]
where the last equality uses both the model and adjoint equations.

The radial and tangential eigenvalues of a radial divergence-free
symmetric tensor satisfy
\[
 \widehat T_{rr}=\frac{2(n-1)}{n-2}\psi'v_C'
            +\frac12(1+C\psi)v_C^{\frac{n+2}{n-2}},
 \qquad
 \widehat T_{\mathrm{tan}}=\widehat T_{rr}+\frac r{n-1}(\widehat T_{rr})'.
\]
In particular, the scalar multiple of $I$ in
\eqref{eq:standalone_stress} contributes to the radial component.
The definition of $\zeta$ gives
$\zeta'=-r\widehat T_{rr}/(n-1)$. Hence the radial and tangential
eigenvalues of $D^2\zeta-(\Delta\zeta)I$ are
\[
 -(n-1)\frac{\zeta'}r=\widehat T_{rr},\qquad
 -\zeta''-(n-2)\frac{\zeta'}r=\widehat T_{\mathrm{tan}}.
\]
The identity extends smoothly through the origin, since the radial
coefficient is smooth and even. Thus $\widehat T=T$, proving
\eqref{eq:standalone_coefficient}; also $\zeta(1)=0$ and
$T_{rr}(1)=-(n-1)\zeta'(1)$.

Finally, the apparent singularity $1/C$ in $\psi$ is removable in
every fixed $C^j(B_2)$ norm: its numerator is smooth in $\alpha$,
vanishes at $\alpha=0$, and
$C=4n(n-1)\alpha/(1+\alpha)^2$. Consequently,
\[
 \psi\longrightarrow\frac{(n+2)(1-r^2)}{8n(n-1)},\qquad
 \zeta\longrightarrow\frac{1-r^2}{4(n-1)},\qquad
 T\longrightarrow\tfrac12 I.
\]
The tensor $T$ is therefore both the exact first variation coefficient
and the formal scalar curvature adjoint applied to $\zeta$:
\[
 L_{\mathrm{sc}}q_1
   =\partial_i\partial_j(q_1)_{ij}-\Delta\tr q_1,
 \qquad L_{\mathrm{sc}}^*\zeta=T.
\]
\end{proof}

\subsection{Moving the domain and cancelling the first variation}

The preceding fixed ball calculation is now applied to the domain
$O=D_f\setminus A$ of Proposition~\ref{prop:marked}. There are two
first order contributions: the metric perturbation and the displacement
of the outer boundary. Scalar flatness converts the former into an area
variation, and the prescribed round boundary cancels their sum.

\begin{figure}[tbp]
\centering
\begin{tikzpicture}[>=Latex,font=\small]
\draw[thick,fill=blue!3] (0,0) circle(1.45);
\fill[gray!50] (.45,.2) ellipse(.35 and .3);
\node at (-.37,-.15) {$\Omega$};
\node at (.54,.75) {$A_0$};
\node at (0,-1.9) {$x\in B$};
\draw[->,thick] (1.9,.35)--(4.0,.35) node[midway,above] {$\Psi=\operatorname{Id}+X_f$};
\node[align=center] at (0,-2.5) {$T(x)$ is built on this fixed ball};
\begin{scope}[xshift=5.9cm]
\path[draw=blue!50!black,thick,fill=blue!3] plot[smooth cycle] coordinates {(-1.48,0) (-.95,-1.2) (.4,-1.35) (1.51,-.45) (1.39,.74) (.43,1.58) (-.85,1.26)};
\fill[gray!50] (.5,.24) ellipse(.39 and .32);
\node at (-.35,-.15) {$O$};
\node at (.62,.82) {$A$};
\node at (0,-1.9) {$\xi\in D_f$};
\node[blue!50!black,anchor=west] at (1.5,1.05) {$r=1+f$};
\end{scope}
\end{tikzpicture}
\caption{The fixed ball used by the fixed ball comparison lemma and the moving target domain. The hole is pulled back together with the metric. Initially the metric contribution is $\int_{\Omega}T(x):q(\Psi(x))\,dx$. Changing variables also changes the tensor's evaluation point and the volume Jacobian; only an energy estimate permits replacing this by $\int_OT(\xi):q(\xi)\,d\xi$ up to $O(\mathscr E)$. The tensor $T$ is generally not constant. The collar map $\Psi$ is distinct from the global replacement $F_0$.}
\label{fig:pullback-domains}
\end{figure}
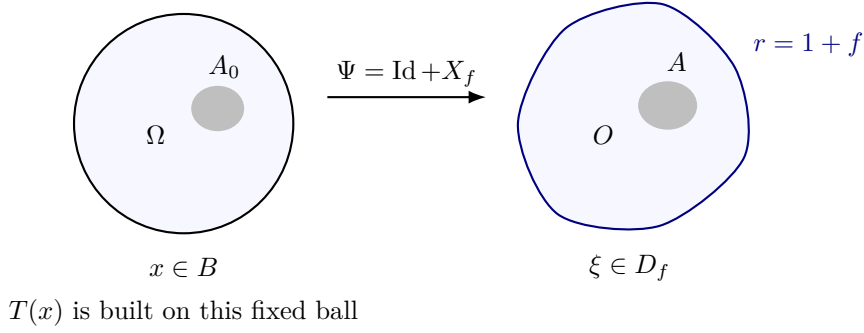

\begin{lemma}\label{lem:unweighted}
With the domain and energy of Proposition~\ref{prop:marked},
\begin{equation}\label{eq:unweighted}
\begin{aligned}
 \left|\int_O w_0^{\frac{n+2}{n-2}}\dd V_h
              -\frac{|\Sph^{n-1}|}{nz}\right|&\le K\mathscr E,\\
 \|w_0-v_C\|_{L^2(O)}&\le KC\sqrt{\mathscr E}.
\end{aligned}
\end{equation}
\end{lemma}

\begin{proof}
\emph{Pullback to the fixed ball.}
Let $\chi_{\mathrm{rad}}$ be a fixed smooth radial cutoff, equal to one
near $r=1$ and supported away from the origin in $B_{3/2}$. Define
\[
 X_f(r\theta):=\chi_{\mathrm{rad}}(r)f(\theta)\theta,
 \qquad \Psi:=\operatorname{Id}+X_f.
\]
For sufficiently small fixed $\tau$, the map $\Psi$ is bi-Lipschitz,
$\Psi(B)=D_f$, and
\[
 \|DX_f\|_\infty\le K\tau,
 \qquad \|X_f\|_{H^1(\R^n)}^2\le K\mathscr E.
\]
This collar map is distinct from the boundary replacement $F_0$:
$F_0$ identifies retained points of the prescribed boundary,
whereas $\Psi$ is introduced solely to apply the fixed ball comparison.
Put $A_0:=\Psi^{-1}(A)$ and $\Omega:=B\setminus A_0$. On $\Omega$,
\begin{equation}\label{eq:pullback-explicit}
\begin{aligned}
 q_1:=\Psi^*(\delta+q)-\delta
 &=q\circ\Psi+DX_f+(DX_f)^T+\mathcal R_\Psi,\\
 \mathcal R_\Psi
 &=(DX_f)^TDX_f+(DX_f)^T(q\circ\Psi)
      +(q\circ\Psi)DX_f\\
 &\quad +(DX_f)^T(q\circ\Psi)DX_f.
\end{aligned}
\end{equation}
Here $q\circ\Psi$ denotes composition of the coordinate components;
the full tensor pullback includes the displayed differential factors.
Cauchy--Schwarz and the fixed amplitude bounds give
\[
 \|\mathcal R_\Psi\|_{L^1(\Omega)}\le K\mathscr E,
 \qquad \|q_1\|_{L^2(\Omega)}^2\le K\mathscr E,
 \qquad \Cap(A_0)\le K\mathscr E.
\]
The remainder is quadratic in the integrated $L^1$ sense used below;
its $L^2$ norm need not be smaller than order $\sqrt{\mathscr E}$.
The capacity estimate follows from the uniform bi-Lipschitz bounds.

The function $w_1:=w_0\circ\Psi$ satisfies the same weak Dirichlet
problem for $h_1:=\delta+q_1$ on $\Omega$, with
$w_1-1\in H^1_0(\Omega)$ and $1\le w_1\le1+KC$.
Thus Lemmas~\ref{lem:solution-taylor} and~\ref{lem:first-variation}
apply. Estimate~\eqref{eq:standalone_state}, together with
$\|Dv_C\|_\infty\|X_f\|_2\le KC\sqrt{\mathscr E}$, proves the
second assertion of \eqref{eq:unweighted}. The first variation formula gives
\[
 \int_O w_0^{\frac{n+2}{n-2}}\dd V_h
       -\frac{|\Sph^{n-1}|}{nz}
 =\int_\Omega T(x):q(\Psi(x))\dd x
   +\int_\Omega T:(DX_f+(DX_f)^T)\dd x+O(\mathscr E).
\]
Crucially, the tensor in the first integral is $T(x)$, evaluated in
the fixed reference ball, not $T(\Psi(x))$. The tensor $T$ is a
smooth radial field and is generally nonconstant.

To return this metric term to the target coordinates $\xi$, the exact
change of variables is
\[
 \int_\Omega T(x):q(\Psi(x))\dd x
 =\int_O
   \frac{T(\Psi^{-1}(\xi)):q(\xi)}
        {\det D\Psi(\Psi^{-1}(\xi))}\dd\xi.
\]
In particular, it is the \emph{inverse} Jacobian that occurs.
Uniform smooth bounds for $T$, the identity
$\det D\Psi=1+\operatorname{div}X_f+O(|DX_f|^2)$, and
Cauchy--Schwarz show that
\[
 \left|\int_\Omega T(x):q(\Psi(x))\dd x-\int_O T:q\dd\xi\right|
 \le K\bigl(\|X_f\|_2+\|DX_f\|_2\bigr)
             \|q\circ\Psi\|_2
 \le K\mathscr E.
\]
Since $\operatorname{div}T=0$, symmetry and integration by parts give
\[
 \int_B T:(DX_f+(DX_f)^T)\dd x
 =2\int_{\partial B}(T\theta)\cdot(f\theta)\dd\sigma
 =2T_{rr}(1)\int_{\Sph^{n-1}}f\dd\sigma.
\]
The omitted integral on $A_0$ is $O(\mathscr E)$ by the boundedness
of $T$ and $DX_f$ and the volume bound for the deleted set. Therefore
\begin{equation}\label{eq:standalone_moving}
 \int_O w_0^{\frac{n+2}{n-2}}\dd V_h
       -\frac{|\Sph^{n-1}|}{nz}
 =2T_{rr}(1)\int_{\Sph^{n-1}}f\dd\sigma
       +\int_O T:q\dd\xi+O(\mathscr E).
\end{equation}
Only first derivatives of $\Psi$ have entered the calculation; no
second derivative control of the Lipschitz boundary graph is needed.

\smallskip\noindent
\emph{Using scalar flatness on the retained domain.}
We work again in target coordinates and use $R_h=0$ on $O$. The
extension $\delta+q_{\mathrm{ref}}$ through $A$ is not assumed scalar flat.
Set
\[
 \mathcal Q:=q-(\tr q)I.
\]
The linearized scalar curvature is
$\operatorname{div}\operatorname{div}\mathcal Q$. More precisely, with
$h=\delta+q$, define
\[
 \Gamma^a_{ij}:=\tfrac12 h^{ab}
     (\partial_iq_{jb}+\partial_jq_{ib}-\partial_bq_{ij}),
 \qquad
 (\mathcal F_{\mathrm{scal}})^a
     :=h^{ij}\Gamma^a_{ij}-h^{ia}\Gamma^b_{ib}.
\]
The exact coordinate formula for scalar curvature yields
\begin{equation}\label{eq:standalone_scalar}
\begin{aligned}
 0&=\operatorname{div}\operatorname{div}\mathcal Q
       +\operatorname{div}\mathcal A_{\mathrm{scal}}
       +\mathcal D_{\mathrm{scal}},\\
 \mathcal A_{\mathrm{scal}}
       &:=\mathcal F_{\mathrm{scal}}-\operatorname{div}\mathcal Q,\\
 |\mathcal A_{\mathrm{scal}}|&\le K|q||Dq|,
 \qquad |\mathcal D_{\mathrm{scal}}|\le K|Dq|^2.
\end{aligned}
\end{equation}
For completeness the scalar remainder is
\[
 \mathcal D_{\mathrm{scal}}
 =-(\partial_bh^{ij})\Gamma^b_{ij}
   +(\partial_jh^{ij})\Gamma^b_{ib}
   +h^{ij}(\Gamma^b_{b\ell}\Gamma^\ell_{ij}
                -\Gamma^b_{j\ell}\Gamma^\ell_{ib}).
\]
This is an exact divergence identity, with no coordinate gauge
condition. It gives
\[
 \operatorname{div}\mathcal F_{\mathrm{scal}}
       =-\mathcal D_{\mathrm{scal}},\qquad
 \|\mathcal F_{\mathrm{scal}}\|_2\le K\sqrt{\mathscr E},\qquad
 \|\mathcal A_{\mathrm{scal}}\|_1
       +\|\mathcal D_{\mathrm{scal}}\|_1\le K\mathscr E.
\]

Choose a capacity cutoff $0\le a_A\le1$, equal to one near $A$, with
$\|a_A\|_{H^1(B_2)}\le K\sqrt{\mathscr E}$, and write
$\chi_A:=1-a_A$. The tensor $\chi_Aq$, extended by zero through $A$,
belongs to $H^1(D_f)$ and has norm at most $K\sqrt{\mathscr E}$:
its derivative is $\chi_ADq+qD\chi_A$, and $q$ has fixed bounded
amplitude. The extension $q_{\mathrm{ref}}$ and the uniform trace
inequality on the small Lipschitz graph domains also imply
\[
 \|q\|_{L^2(\Gamma)}\le K\sqrt{\mathscr E},
 \qquad \Gamma:=\partial D_f\setminus A.
\]
The trace constant is uniform because the boundary graphs have a
common Lipschitz bound and a fixed size finite coordinate cover.

Although $\zeta(1)=0$, the function $\zeta$ need not vanish on
$\partial D_f$. Subtract its actual boundary value by defining
\[
 \zeta_{\partial}(r,\theta)
       :=\chi_{\mathrm{rad}}(r)\zeta(1+f(\theta)).
\]
Then $\zeta-\zeta_{\partial}$ has zero outer trace, and
\[
 \|\zeta_{\partial}\|_{H^1(D_f)}\le K\sqrt{\mathscr E},
 \qquad \|\zeta_{\partial}\|_\infty\le K.
\]
Use $\chi_A(\zeta-\zeta_{\partial})\in H^1_0(O)$ as a weak test
in $\operatorname{div}\mathcal F_{\mathrm{scal}}
=-\mathcal D_{\mathrm{scal}}$. This test is bounded and vanishes
near $A$. Its use against the $L^1$ remainder is justified by uniformly
bounded smooth approximants converging strongly in $H^1$ and almost
everywhere; such approximants are obtained by a fixed bounded
truncation of the usual $H^1_0$ approximants. Consequently,
\[
 \int_O\mathcal F_{\mathrm{scal}}\cdot D(\chi_A\zeta)
 =\int_O\mathcal F_{\mathrm{scal}}\cdot D(\chi_A\zeta_{\partial})
   +\int_O\mathcal D_{\mathrm{scal}}\chi_A
                (\zeta-\zeta_{\partial})
 =O(\mathscr E).
\]
In particular, the derivative falling on $\chi_A$ is paired with
the small $L^2$ flux; no small supremum norm of $\zeta_{\partial}$
is required.

The identity $T:q=\mathcal Q:D^2\zeta$ and integration by parts now give
\begin{align*}
 \int_O\chi_A T:q
 ={}&\int_\Gamma\chi_A\mathcal Q(\nu,D\zeta)
       -\int_O\mathcal F_{\mathrm{scal}}\cdot D(\chi_A\zeta)
       +\int_O\zeta\mathcal F_{\mathrm{scal}}\cdot D\chi_A\\
 &+\int_O\chi_A\mathcal A_{\mathrm{scal}}\cdot D\zeta
       -\int_O\mathcal Q(D\chi_A,D\zeta).
\end{align*}
Every interior term is $O(\mathscr E)$: use the preceding tested
identity, the $L^1$ bound for $\mathcal A_{\mathrm{scal}}$, and the
products of the small $L^2$ norms with $\|D\chi_A\|_2$.
Removing $\chi_A$ in the volume and boundary integrals costs at most
\[
 K\|a_A\|_{L^2(D_f)}\|q\|_{L^2(O)}
 +K\|a_A\|_{L^2(\partial D_f)}\|q\|_{L^2(\Gamma)}
 \le K\mathscr E.
\]
Here and below boundary integrals without a metric label use Euclidean
area. We obtain
\begin{equation}\label{eq:standalone_boundary}
 \int_O T:q
 =\int_\Gamma\mathcal Q(\nu,D\zeta)+O(\mathscr E)
 =-\zeta'(1)\int_\Gamma\tr_\Gamma q+O(\mathscr E).
\end{equation}
The normal $\nu$ is Euclidean, and $\tr_\Gamma$ is Euclidean tangential
trace. To verify the second equality, decompose
\[
 \mathcal Q(\nu,D\zeta)
 =-(\partial_\nu\zeta)\tr_\Gamma q
       +q(\nu,D_\Gamma\zeta).
\]
On the radial graph,
$D_\Gamma\zeta=O(|Df|)$ and
$\partial_\nu\zeta=\zeta'(1)+O(|f|+|Df|^2)$.
The terms linear in $f$ or $Df$ pair with the small $L^2$ trace of $q$;
the remaining term is bounded by $K\|q\|_\infty\|Df\|_2^2$.
All are $O(\mathscr E)$. The integration by parts uses the full
Lipschitz domain $D_f$ and the extended tensor $\chi_A\mathcal Q$.
It requires no normal derivative trace of $q$ and no regularity of
the boundary of $A$.

\smallskip\noindent
\emph{Cancellation by the boundary area.}
Proposition~\ref{prop:marked} gives
$\operatorname{Area}_h(\Gamma)=|\Sph^{n-1}|+O(\mathscr E)$.
The tangential determinant and radial graph expansions give, respectively,
\[
 \operatorname{Area}_h(\Gamma)
 =|\Gamma|+\tfrac12\int_\Gamma\tr_\Gamma q+O(\mathscr E),
 \qquad
 |\partial D_f|
 =|\Sph^{n-1}|+(n-1)\int_{\Sph^{n-1}}f\dd\sigma
       +O(\mathscr E).
\]
Their remainders are bounded by $K\|q\|_{L^2(\Gamma)}^2$ and
$K\|f\|_{H^1(\Sph^{n-1})}^2$. The omitted target boundary also
has area $O(\mathscr E)$, so
\[
 \int_\Gamma\tr_\Gamma q
       +2(n-1)\int_{\Sph^{n-1}}f\dd\sigma=O(\mathscr E).
\]
Combining this identity with \eqref{eq:standalone_moving} and
\eqref{eq:standalone_boundary}, the coefficient of $\int f$ becomes
\[
 2T_{rr}(1)+2(n-1)\zeta'(1)=0
\]
by \eqref{eq:standalone_stresspotential}. This proves the first
assertion of \eqref{eq:unweighted}.
\end{proof}

\subsection{The harmonic weight}

The exact identity \eqref{eq:Jweighted} contains the positive factor
$v^{-1}$, which need not be at least one in the fill-in. Its deviation
from one has a nonnegative \emph{boundary mean}, however. We transfer
that sign to the required interior integral by a radial Poisson test
whose normal derivative is constant on the reference sphere.

\begin{lemma}\label{lem:weighted}
The harmonic correction satisfies
\begin{equation}\label{eq:harmonicweighted}
 \int_O(v^{-1}-1)w_0^{\frac{n+2}{n-2}}\dd V_h
       \ge-K\mathscr E.
\end{equation}
Consequently,
\[
 J\ge\frac{|\Sph^{n-1}|}{n}-K(m_\gamma+\varepsilon).
\]
\end{lemma}

\begin{proof}
In the selected coordinates, $v^{-1}$ is $h$-harmonic on $O$, and
Proposition~\ref{prop:marked} gives
\[
 \|v^{-1}-1\|_{H^1(O)}+\|v^{-1}-1\|_{L^2(\Gamma)}
       +\|q\|_{L^2(O)}\le K\sqrt{\mathscr E}.
\]
Moreover $|v^{-1}-1|\le K$, since $0<v^{-1}\le z^{-1}$.
We assume $\mathscr E\le1$, as ensured by choosing $C$ and the
approximation error sufficiently small.

Choose a capacity cutoff $0\le a_A\le1$, equal to one near $A$, with
\[
 \|a_A\|_{H^1(B_2)}\le K\sqrt{\mathscr E}.
\]
An arbitrarily small excess in this bound may be sent to zero after
the estimates. The function $(1-a_A)(v^{-1}-1)$ extends by zero
through $A$ to an $H^1(D_f)$ function. Integrals of this expression
over $D_f$ always mean that extension; no value of $v^{-1}$ is
assigned to points without a selected preimage.

Define, for $0\le r\le2$,
\[
 \psi_0(r):=\int_r^1 t^{1-n}
    \left(\int_0^t s^{n-1}v_C(s)^{\frac{n+2}{n-2}}\dd s\right)\dd t.
\]
This function is smooth at zero and satisfies
\[
 -\Delta\psi_0=v_C^{\frac{n+2}{n-2}}\quad\text{on }B_2,
 \qquad \psi_0(1)=0,
 \qquad -\psi_0'(1)
       =\frac1{|\Sph^{n-1}|}\int_Bv_C^{\frac{n+2}{n-2}}\dd x
       =\frac1{nz}.
\]
Its required derivatives are uniformly bounded for small $C$.
Correct its boundary value on $\partial D_f$ by
\[
 \psi_f(r,\theta):=\psi_0(r)
       -\chi_{\mathrm{rad}}(r)\psi_0(1+f(\theta)).
\]
Then $\psi_f$ has zero outer trace and
\[
 \|\psi_0-\psi_f\|_{H^1(D_f)}\le K\sqrt{\mathscr E},
 \qquad \|\psi_f\|_{W^{1,\infty}(D_f)}\le K.
\]
These bounds follow from $\psi_0(1)=0$ and the $H^1$ and fixed
Lipschitz bounds for $f$.

The function $(1-a_A)\psi_f\in H^1_0(O)$ is an admissible test in
the weak harmonic equation for $v^{-1}$. It vanishes near $A$ and
has zero trace on the Lipschitz outer boundary. Testing gives
\[
 \int_O(1-a_A)
       \langle\nabla^h(v^{-1}),\nabla^h\psi_f\rangle_h\dd V_h
 =\int_O\psi_f
       \langle\nabla^h(v^{-1}),\nabla^h a_A\rangle_h\dd V_h
 =O(\mathscr E).
\]
Replacing the left pairing and measure by their Euclidean versions
costs at most
\[
 K\|q\|_2\|D(v^{-1})\|_2\|D\psi_f\|_\infty
\le K\mathscr E.
\]
Replacing $\psi_f$ by $\psi_0$ costs at most
\[
 K\|D(v^{-1})\|_2\|D(\psi_0-\psi_f)\|_2
\le K\mathscr E.
\]
The product rule therefore gives
\[
 \left|\int_{D_f}D[(1-a_A)(v^{-1}-1)]\cdot D\psi_0\dd\xi\right|
 \le K\mathscr E
       +K\|v^{-1}-1\|_2\|Da_A\|_2
 \le K\mathscr E.
\]
The small $L^2$ norm of $v^{-1}-1$ is essential in the last estimate;
its uniform bound alone would leave a square root of the capacity.

Integrating by parts on the full Lipschitz domain $D_f$ yields
\begin{equation}\label{eq:radialtest}
\begin{aligned}
 \int_{D_f}(1-a_A)(v^{-1}-1)
             v_C^{\frac{n+2}{n-2}}\dd\xi
 ={}&-\int_{\partial D_f}(1-a_A)(v^{-1}-1)
                   \partial_\nu\psi_0\dd A_\delta\\
 &+O(\mathscr E).
\end{aligned}
\end{equation}
The normal $\nu$ and the area element $\dd A_\delta$ are Euclidean.
This integration uses the zero extension above, so no artificial
inner boundary integral arises.

Pull the outer boundary integral back to the prescribed unit sphere
using $F_0$. Its tangential Jacobian $J_{F_0}$ satisfies
\[
 \left\|-\bigl((\partial_\nu\psi_0)\circ F_0\bigr)J_{F_0}
                    -\frac1{nz}\right\|_{L^2(\Sph^{n-1})}
       \le K\sqrt{\mathscr E}.
\]
Indeed, the radius, unit normal, and tangential Jacobian of the
parametrization $F_0|_{\Sph^{n-1}}$ differ in $L^2$ from the
corresponding quantities of the unit sphere by at most
$K\|F_0-\operatorname{Id}\|_{W^{1,2}(\Sph^{n-1})}$.
The radial function $\psi_0$ has uniformly bounded derivatives.

Let $\Sigma_\Gamma:=F_0^{-1}(\Gamma)\subset\Sigma$. On this set,
$F_0$ agrees with the selected boundary coordinates, so
the pulled-back values of $v^{-1}$ are its original boundary values.
The $L^2$ coefficient error just estimated pairs with the
$O(\sqrt{\mathscr E})$ trace of $v^{-1}-1$, giving $O(\mathscr E)$.
For the constant coefficient, \eqref{eq:inverse-v-mean} gives
\[
 \int_\Sigma(v^{-1}-1)\dd\sigma
       =\frac{|\Sph^{n-1}|m_\gamma}{2z^2}\ge0.
\]
Omitting $\Sigma\setminus\Sigma_\Gamma$ costs $O(\mathscr E)$
because this set has area $O(\mathscr E)$ and $v^{-1}-1$ is bounded.
The cutoff contribution also satisfies
\[
 \left|\int_{\Sigma_\Gamma}
       (a_A\circ F_0)(v^{-1}-1)\dd\sigma\right|
 \le K\|a_A\|_{H^1(B_2)}
             \|v^{-1}-1\|_{L^2(\Sigma_\Gamma)}
 \le K\mathscr E,
\]
by the uniform trace inequality and bounded area distortion.
Thus the right-hand side of \eqref{eq:radialtest} is at least
$-K\mathscr E$.

Removing $1-a_A$ in the volume integral costs at most
$K\|a_A\|_2\|v^{-1}-1\|_2\le K\mathscr E$.
The positive functions $w_0$ and $v_C$ remain in a fixed bounded
interval. By \eqref{eq:unweighted}, replacing
$v_C^{(n+2)/(n-2)}$ by $w_0^{(n+2)/(n-2)}$ costs at most
\[
 K\|v^{-1}-1\|_2\|w_0-v_C\|_2
       \le KC\mathscr E\le K\mathscr E.
\]
Finally, replacing $\dd\xi$ by $\dd V_h$ costs at most
$K\|q\|_2\|v^{-1}-1\|_2\le K\mathscr E$.
This proves \eqref{eq:harmonicweighted}.

Both factors $v^{-1}$ and $V$ are positive on the fill-in.
The domain $O$ represents only a subset of that fill-in, so dropping
the remaining positive integral in \eqref{eq:Jweighted} and using
$V\ge w_0$ on $O$ gives
\[
 J=z\int_M v^{-1}V^{\frac{n+2}{n-2}}\dd V_h
 \ge z\int_O v^{-1}w_0^{\frac{n+2}{n-2}}\dd V_h
 \ge\frac{|\Sph^{n-1}|}{n}-K\mathscr E.
\]
The conclusion follows from $\mathscr E\le K(m_\gamma+\varepsilon)$.
\end{proof}

\subsection{Absorbing the mass}

\begin{proposition}\label{prop:mass-vanishing}
For all sufficiently small $C>0$, depending only on $n$, the residual
ADM mass satisfies $m_\gamma=0$.
\end{proposition}

\begin{proof}
Fix the dimensional graph slope, then the metric amplitude $\tau$
required by ellipticity, semilinear coercivity, and the Taylor
estimates. Choose the exterior displacement tolerance $\eta$ in terms
of $\tau$, next $\beta$ small relative to $\eta^2$, and finally the
fixed potential thresholds of the coordinate construction. These
choices determine a finite constant $K_n\ge1$ with
\[
 J\ge\frac{|\Sph^{n-1}|}{n}
          -K_n(m_\gamma+\varepsilon).
\]
Choose $0<C_{\mathrm{pre}}(n)<n(n-1)$ so that the model and comparison
solutions exist and the a priori bound
$m_\gamma\le C/[2n(n-1)]$ enforces the remaining absolute
smallness requirements for $C\le C_{\mathrm{pre}}$ and sufficiently
small $\varepsilon$. This includes the volume bound used to choose
the chain of exterior points in Lemma~\ref{lem:marked-boundary}.

The mass budget implies
\[
 \bigl[2(n-1)|\Sph^{n-1}|-CK_n\bigr]m_\gamma
       \le CK_n\varepsilon.
\]
It therefore suffices to take
\begin{equation}\label{eq:qualitative-threshold}
 C_*(n):=\min\left\{C_{\mathrm{pre}}(n),
           \frac{(n-1)|\Sph^{n-1}|}{K_n}\right\}.
\end{equation}
For each original fill-in and each fixed $0<C\le C_*(n)$, the
numbers $m_\gamma$ and $J$ are fixed, whereas $\varepsilon$ may be
arbitrarily small. The bracket is at least $(n-1)|\Sph^{n-1}|>0$;
letting $\varepsilon\downarrow0$ proves $m_\gamma=0$. No convergence
of the coordinate maps, domains, or deleted sets is required. This argument establishes a positive dimensional threshold; the
numerical choice stated after Theorem~\ref{thm:main} requires tracking
the constants, which is omitted here.
\end{proof}

\section{Equality and the topology of the fill-in}\label{sec:equality}

The argument now recovers the topology as part of the equality case.
First the scalar flat fill-in becomes a Euclidean ball; then a positive
static potential identifies its original conformal metric.

By Proposition~\ref{prop:mass-vanishing} and \eqref{eq:energybound},
\[
 \int_N|dv|_\gamma^2\dd V_\gamma=0.
\]
Connectedness, continuity across $\Sigma$, and normalization at infinity
imply $v\equiv1$. Equation~\eqref{eq:uv} makes $\gamma$ Euclidean
on the attached exterior. Its two mean curvatures agree at $\Sigma$;
after the rescaling in \eqref{eq:hgamma}, the exterior is Euclidean
with unit spherical boundary. Hence
\[
 R_h=0,\qquad h_\Sigma=\sigma,\qquad H_h=n-1.
\]
Lemma~\ref{lem:EuclideanBYrigidity} now identifies $(M,h)$ with
$(\overline B,\delta)$. 

Under this identification,
$g=V^{4/(n-2)}\delta$. Let
$g_C:=v_C^{4/(n-2)}\delta$ be the round cap metric of
Lemma~\ref{lem:round-model}, and define
\[
 \varrho_{\mathrm{cap}}:=\frac{V}{v_C},
 \qquad g=\varrho_{\mathrm{cap}}^{\frac4{n-2}}g_C.
\]
Both conformal factors equal one on the boundary, so
$\varrho_{\mathrm{cap}}|_{\partial B}=1$. The original scalar curvature and
mean curvature hypotheses give
\[
 -\frac{4(n-1)}{n-2}\Delta_{g_C}\varrho_{\mathrm{cap}}
       +C\varrho_{\mathrm{cap}}
       \ge C\varrho_{\mathrm{cap}}^{\frac{n+2}{n-2}},
 \qquad
 \partial_{\nu_C}\varrho_{\mathrm{cap}}\ge0
       \quad\text{on }\partial B.
\]
Here $\nu_C$ is the outward $g_C$-unit normal. The boundary inequality
uses
\[
 H_g=H_C+\frac{2(n-1)}{n-2}
                  \partial_{\nu_C}\varrho_{\mathrm{cap}}.
\]
No pointwise comparison of $\varrho_{\mathrm{cap}}$ with one is needed.

The round cap has the positive static potential
\[
 f_C(r):=\frac{1-\alpha r^2}{1+\alpha r^2},
 \qquad \alpha=\frac{1-z}{z}\in(0,1),
 \qquad
 \Delta_{g_C}f_C=-\frac{C}{n-1}f_C.
\]
Test the scalar inequality with $f_C$ and apply Green's identity to
$\varrho_{\mathrm{cap}}-1$, which has zero boundary trace. This yields
\begin{align*}
 0
 &\le\int_B f_C\left[
      -\frac{4(n-1)}{n-2}\Delta_{g_C}\varrho_{\mathrm{cap}}
      +C\varrho_{\mathrm{cap}}
      -C\varrho_{\mathrm{cap}}^{\frac{n+2}{n-2}}
                         \right]\dd V_{g_C}\\
 &=-\frac{4(n-1)}{n-2}\int_{\partial B}
      f_C\partial_{\nu_C}\varrho_{\mathrm{cap}}\dd\sigma\\
 &\quad-C\int_B f_C\left[
      \varrho_{\mathrm{cap}}^{\frac{n+2}{n-2}}-1
       -\frac{n+2}{n-2}(\varrho_{\mathrm{cap}}-1)
                         \right]\dd V_{g_C}
 \le0.
\end{align*}
The coefficient produced by the static equation is $4C/(n-2)$,
which combines with the linear scalar term to give the displayed
tangent line remainder. The boundary term is nonpositive, and strict
convexity of $s\mapsto s^{(n+2)/(n-2)}$ for $s>0$ makes the second
integrand nonnegative, vanishing only at $s=1$. Since $C>0$ and
$f_C>0$ on $\overline B$, it follows that
$\varrho_{\mathrm{cap}}\equiv1$. Thus $V=v_C$ and $g=g_C$.

This proves Theorem~\ref{thm:main}. Rescaling by
$C=n(n-1)\sin^2r_0$ gives Corollary~\ref{cor:gromov}.

\end{document}